\documentclass[12pt,leqno,twoside]{amsart}

\usepackage{amssymb,amsmath,amsthm,soul,color}

\usepackage[noadjust]{cite}

\usepackage{t1enc}

\usepackage[utf8]{inputenc}
\usepackage{indentfirst}

\usepackage{todonotes}

\usepackage[left=2.5cm, right=2.5cm, top=2cm, bottom=2cm]{geometry}

\usepackage{url}

\usepackage{mathrsfs}

\usepackage{enumitem,graphicx}
\usepackage[colorlinks=true,urlcolor=blue,
citecolor=red,linkcolor=blue,linktocpage,pdfpagelabels,
bookmarksnumbered,bookmarksopen]{hyperref}
\usepackage[metapost]{mfpic}
\usepackage[normalem]{ulem}

\usepackage{mathtools}

\usepackage{empheq}

\usepackage{braket} 

\newtheorem{Th}{Theorem}[section]
\newtheorem{Prop}[Th]{Proposition}
\newtheorem{Lem}[Th]{Lemma}

\newtheorem{Rem}[Th]{Remark}

\newenvironment{altproof}[1]
{\noindent
{\em Proof of {#1}}.}
{\nopagebreak\mbox{}\hfill $\Box$\par\addvspace{0.5cm}}

\newcommand{\vp}{\varphi}

\newcommand{\eps}{\varepsilon}

\newcommand{\R}{\mathbb{R}}

\newcommand{\N}{\mathbb{N}}

\newcommand{\cC}{{\mathcal C}}
\newcommand{\cD}{{\mathcal D}}

\newcommand{\cF}{{\mathcal F}}
\newcommand{\cG}{{\mathcal G}}
\newcommand{\cH}{{\mathcal H}}

\newcommand{\cM}{{\mathcal M}}

\newcommand{\cO}{{\mathcal O}}

\newcommand{\cQ}{{\mathcal Q}}

\newcommand{\cS}{{\mathcal S}}

\newcommand{\cX}{{\mathcal X}}

   \newcommand\D{\mathcal{D}}

\newcommand{\weakto}{\rightharpoonup}

\newcommand{\tu}{\widetilde{u}}
\newcommand{\tv}{\widetilde{v}}

  \newcommand{\curl}{\nabla \times}
  \renewcommand{\div}{\mathrm{div}\,}

\numberwithin{equation}{section}

\newcommand{\supp}{\mathrm{supp}\,}

\newcommand{\rn}{\R^N}
\newcommand{\UU}{\mathbf{U}}

\newcommand{\SO}{\cS\cO}

\newcommand{\codim}{\mathrm{co}\,\mathrm{dim}\,}

\begin{document}


\title{Normalized solutions to at least mass critical problems: singular polyharmonic equations and related curl-curl problems}

\author[B. Bieganowski]{Bartosz Bieganowski}
	\address[B. Bieganowski]{\newline\indent
			Faculty of Mathematics, Informatics and Mechanics, \newline\indent
			University of Warsaw, \newline\indent
			ul. Banacha 2, 02-097 Warsaw, Poland}	
			\email{\href{mailto:bartoszb@mimuw.edu.pl}{bartoszb@mimuw.edu.pl}}

\author[J. Mederski]{Jaros\l aw Mederski}
	\address[J. Mederski]{\newline\indent
			Institute of Mathematics,
			\newline\indent 
			Polish Academy of Sciences,
			\newline\indent 
			ul. \'Sniadeckich 8, 00-656 Warsaw, Poland
	}
	\email{\href{mailto:jmederski@impan.pl}{jmederski@impan.pl}}

\author[J. Schino]{Jacopo Schino}
\address[J. Schino]{\newline\indent
	Department of Mathematics,
	\newline\indent 
	North Carolina State University,
	\newline\indent 
	2311 Stinson Drive, 27607 Raleigh, NC, USA
	\newline\indent
	and
	\newline\indent
	Faculty of Mathematics, Informatics and Mechanics,
	\newline\indent 
	University of Warsaw,
	\newline\indent 
	ul. Banacha 2, 02-097 Warsaw, Poland
}
\email{\href{mailto:j.schino2@uw.edu.pl}{j.schino2@uw.edu.pl}}		
	
	\date{}	\date{\today} \maketitle
	
	\pagestyle{myheadings} \markboth{\underline{B. Bieganowski, J. Mederski, J. Schino}}{
		\underline{Normalized solutions to at least mass critical problems}}

\begin{abstract} 
We are interested in the existence of normalized solutions to the problem
\begin{equation*}
\begin{cases}
(-\Delta)^m u+\frac{\mu}{|y|^{2m}}u + \lambda u =  g(u), \quad x = (y,z) \in \R^K \times \R^{N-K}, \\
\int_{\R^N} |u|^2 \, dx = \rho > 0,
\end{cases}
\end{equation*}
in the so-called at least mass critical regime. We utilize recently introduced variational techniques involving the minimization on the $L^2$-ball. Moreover, we find also a solution to the related curl-curl problem
\begin{equation*}
	\begin{cases}
		\nabla\times\nabla\times\UU+\lambda\UU=f(\UU), \quad x \in \R^N,\\
		\int_{\rn}|\UU|^2\,dx=\rho,\\
	\end{cases}
\end{equation*}
which arises from the system of Maxwell equations and is of great importance in nonlinear optics.

\medskip

   
\end{abstract}

\maketitle

\section{Introduction}

In this paper, we look for weak solutions $u\colon\R^N\to\R$ and $\UU\colon\R^N\to\R^N$ to the following constrained problems
\begin{eqnarray}\label{eq}
&&\begin{cases}
(-\Delta)^m u + \frac{\mu}{|y|^{2m}}u + \lambda u =  g(u), \quad x = (y,z) \in \R^K \times \R^{N-K},\\
\int_{\R^N} |u|^2 \, dx = \rho > 0,
\end{cases}\\
&&\label{e-normcurl}
\begin{cases}
\nabla\times\nabla\times\UU+\lambda\UU=f(\UU), \quad x \in \R^{N},\\
\int_{\rn}|\UU|^2\,dx=\rho,
\end{cases}
\end{eqnarray}
where $N \ge K \ge 2m$, $\mu \in \R$ is a real parameter of the {\em singular potential} $V(x):=\frac{\mu}{|y|^{2m}}$, and $g \colon \R\to\R$ and $f \colon \R^N\to\R^N$ are {\em nonlinear terms} that will be discussed later (for the definition of $\nabla \times \nabla \times \UU$ when $N\ge3$, see Section \ref{sec:curlcurl}). Moreover, the \textit{mass} $\rho$ is prescribed, while $\lambda$ is unknown and will arise as a Lagrange multiplier. 

The first problem appears naturally when one studies the Cauchy problem for the time-dependent \textit{polyharmonic} equation
\begin{equation}\label{schroed-t}
\begin{cases}
\mathbf{i} \partial_t \Psi = (-\Delta)^m \Psi + \frac{\mu}{|y|^{2m}}\Psi  - \mathbf{f}(|\Psi|)\Psi, \\
\Psi(0, \cdot) = \psi_0 \in L^2 (\R^N) \setminus \{ 0\},
\end{cases}
\end{equation}
which was considered in \cite{IK,Turitsyn} with $m=2$ to study the stability of solitons in magnetic materials once the effective quasiparticle mass becomes infinite. Note that, formally, \eqref{schroed-t} preserves the $L^2$-norm of the solution. Indeed multiplying \eqref{schroed-t} by $\overline{\Psi}$, integrating, and taking the imaginary part leads to $\frac{d}{dt} \int_{\R^N} |\Psi|^2 \, dx = 0$ and therefore we can define $\rho := \int_{\R^N} |\Psi|^2 \, dx = \int_{\R^N} |\psi_0|^2 \, dx > 0$. The ansatz $\Psi(x, t) = e^{i\lambda t} u(x)$ leads then to \eqref{eq}. The conserved mass is particularly meaningful from the physical point of view; for instance, when $m=1$, it represents the total number of atoms in Bose--Einstein condensation \cite{Malomed}, or
the total power in nonlinear optics \cite{Buryak}. In the latter case, the nonlinear Schr\"odinger equation of the form \eqref{eq} models the propagation of electromagnetic waves travelling through a waveguide \cite{Buryak,Agrawal}. The transport of electromagnetic waves in nonlinear dielectric media is controlled exactly by the nonlinear Maxwell equations, which are challenging to solve analytically. Therefore, by means of some simplifications and approximations, for instance, {\em  scalar approximation} or {\em slowly varying envelope approximation}, the problem \eqref{eq} is intensively studied in nonlinear optics. We emphasize, however, that such approximations may produce {\em non-physical} solutions; see \cite{Akhmediev-etal, Ciattoni-etal:2005} and the references therein.

In this paper, we show that, if $\mu=1$, $K = 2$, $m=1$, $f$ and $g$ are odd and related by
\begin{equation}\label{e-fgh}
|\mathbf{V}|f(\mathbf{V}) = g(|\mathbf{V}|) \mathbf{V}, \quad \mathbf{V} \in \R^N,
\end{equation}
and $u$ solves \eqref{eq}, then
\begin{equation*}
\mathbf{U}(x) = \frac{u(x)}{r} \left( \begin{array}{c}
-x_2 \\ x_1 \\ 0
\end{array} \right), \quad r = \sqrt{x_1^2+x_2^2}
\end{equation*}
solves \eqref{e-normcurl}. In particular, if $N=3$, we find indeed the {\em exact propagation} of time harmonic electromagentic waves, where the electric field is given by $E(x,t)=\UU(x)\cos(\omega t)$ and  Maxwell's equations together with the material constitutive laws are satisfied \cite{Agrawal,BartschMederski1,Stuart:1993,McLeod}. The model nonlinearity is $f(\UU)=|\UU|^2\UU$, hence $g(u)=|u|^2u$, which represents the {\em Kerr effect} in a nonlinear medium. In general, we focus on mass-critical and mass-supercritical problems with
\begin{equation}\label{example}
g(u)=\eta_1 |u|^{2_*-2}u + \eta_2 |u|^{p-2}u, \quad \eta_1\geq 0,\eta_2>0
\end{equation}
as a model in mind, where $2_{*} := 2+\frac{4m}{N}<p<2^*$, $2^*=\frac{2N}{N-2m}$ for $N>2m$, and $2^*=\infty$ for $N=2m$.

Recall that normalized Schr\"odinger equations (i.e., $m=1$) in the mass critical or mass super critical cases have been studied, for instance, in  \cite{JeanjeanLuNorm, BartschMolle, BartschSoaveJFA, BartschSoaveJFACorr, BartschJS2016, Soave, Jeanjean, JeanjeanLe, Schino, CSz}, while the curl-curl problem \eqref{e-normcurl} has not been investigated yet. The polyharmonic case, on the other hand, seems to have been studied only in \cite{Ma_Chang,Phan} and uniquely for $m=2$, in contexts that do not overlap with those debated here. Often in these works, the radial symmetry plays an important role or some particular assumptions must be imposed on the potentials $V$. In the latter case there are only few works; for instance, Noris, Tavares, and Verzini consider a class of trapping potentials, i.e., $\lim_{|x|\to\infty}V(x)=+\infty$, in \cite{Noris}, and Bartsch et. al. assume that the $L^{\infty}$- or $L^{N/2}$-norm of $V$ is sufficiently small  in \cite{BartschMolle}, see also \cite{Ding,Molle}. Clearly,  the singular potential $V(x)=\frac{\mu}{|y|^{2}}$ does not lie in these classes and, in addition, our approach is quite different. Moreover, we do not assume any smallness property on $V$, at least when $\mu>0$. The only piece of work about normalized solutions in the presence of a singular potential, to the best of our knowledge, is \cite{Li_Zou}, which, however, not only studies a different regime for the nonlinear term, but -- additionally -- only consider the spherical case (i.e., $N=K$) for the Schr\"odinger equation (i.e., $m=1$) in dimension $N > 2m = 2$. Let us mention also that the unconstrained problems, i.e., \eqref{eq} and \eqref{e-normcurl} without $L^2$-constraints, have been intensively studied, see for instance \cite{BadBenRol,GMS,LLT,MederskiSchinoSurvey} and the references therein.

Our strategy is to refine and extend a recent minimization technique, which has been introduced in \cite{BM} for the problem \eqref{eq} with $m=1$, $\mu = 0$, and $N \ge 3$, extended in \cite{MS} for systems of equations with Sobolev-critical nonlinearities and in \cite{CLY} for exponential critical ones, and adapted in \cite{Schino} to the mass subcritical case. As we will see in Theorem \ref{th:generalization} below, if $\mu = 0$ and $N > 2m$, then there is a {\em normalized ground state solution} $u$ to \eqref{eq}, which is radial and positive if $m=1$, i.e., $u$ solves \eqref{eq} and $J_0(u)=\inf_{\cM_0\cap\cS} J_0$, where
\begin{align}\label{intro:eq1}
J_0(u)&:=\frac12 \int_{\R^N} |\nabla^{m} u|^2\, dx - \int_{\R^N} G(u) \, dx,\\
\label{intro:M0}
\cM_0&:=\left\{ u \in H^{m}(\R^N) \setminus \{0\} \ : \  \int_{\R^N} |\nabla^m u|^2 \, dx=\frac{N}{2m} \int_{\R^N} H(u) \, dx \right\}, \\
\label{def:nabla-m}
\nabla^m u &:=
\begin{cases}
\Delta^{m/2} u & \text{if } m  \text{ is even,}\\
\nabla \Delta^{(m-1)/2} u & \text{if } m \text{ is odd,}
\end{cases}
\end{align}
$\cS$ is the sphere in $L^2(\R^N)$ centred at $0$ with radius $\sqrt{\rho}$, and $H(u) := g(u)u - 2 G(u)$, $G(u)=\int_0^ug(s)\,ds$. The nonlinearity $g$ may satisfy general growth conditions that cover the example \eqref{example} -- see, e.g., \cite{BM,JeanjeanLuNorm}. 
We would like to emphasize that if $\mu>0$, then there are no normalized ground state solutions to \eqref{eq} with the natural extension of the definition of the constraint $\cM_0$, which will be demonstrated in Theorem \ref{thm:nonexistence} below.

In order to solve \eqref{eq} with $\mu\neq 0$ we use the scaling property of the Laplace operator involving the singular potential and we show that every nontrivial weak solution to \eqref{eq} invariant with respect to the action of $\cG(K) := \cO(K) \times \mathrm{id}_{N-K}$ lies in the following constraint 
\begin{equation}\label{def:M}
\cM := \left\{ u \in H^{m}_{\cG(K)} (\R^N) \setminus \{0\} \ : \  \int_{\R^N} |\nabla^{m} u|^2 + \frac{\mu}{|y|^{2m}} u^2 \, dx=\frac{N}{2m} \int_{\R^N} H(u) \, dx, \right\},
\end{equation}
where  $ H^{m}_{\cG(K)} (\R^N)$ stands for the $\cO(K) \times \mathrm{id}_{N-K}$-invariant functions of $H^{m}(\R^N)$. The {\em energy functional} 
\begin{equation}\label{def:J}
J(u) = \frac12 \int_{\R^N} |\nabla^{m} u|^2 + \frac{\mu}{|y|^{2m}} u^2 \, dx - \int_{\R^N} G(u) \, dx
\end{equation}
is coercive and bounded from below on $\cM\cap\D$, where
$$
\cD := \left\{ u \in L^2(\R^N)\ : \ \int_{\R^N} |u|^2 \, dx \leq \rho \right\}.
$$
Our principal aim is to show that $\inf_{\cD \cap \cM} J$ is attained by a solution to \eqref{eq}. The role of the symmetry is twofold: it is required in the proof that every critical point of $J$ satisfies the Poho\v{z}aev identity, hence lies in $\cM$ (Section \ref{sec:Poho}), and it is essential for the existence of minimizers of $J$ (cf. Theorem \ref{thm:nonexistence}). We underline that the Poho\v{z}aev identity, on which our existence results Theorems \ref{th:main} and \ref{T:Vec} heavily rely, seems to be new in the cylindrical case $N>K$ and in the polyharmonic case with $m\ge3$, and therefore we consider it among the main outcomes of this paper, see Proposition \ref{poh:local} and Remark \ref{R:Smets} below. It is worth mentioning that  we cannot apply the classical Pucci--Serrin argument here \cite{PucciSerrin} due to the lack of global regularity of solutions and the presence of the singular potential. See also \cite{Gazzola}, where the Poho\v{z}aev identity has been obtained for weak solutions only in a particular case, e.g., on bounded domains with a power-type nonlinearity.

Finally, we point out that Theorem \ref{th:main} with $m=1$ improves \cite[Theorem 1.1]{BM} also in the autonomous case $\mu=0$: not only do we relax the smallness assumption about $\rho$ when $g$ has a mass critical growth at the origin (cf. \eqref{eq:Hstrict} and \cite[formula (1.5)]{BM}), but we can also deal with the case $N=2$, which will often require tailored techniques.

The paper is organized as follows. In Section \ref{sec:setting}, we build the functional setting and we prove the Poho\v{z}aev identity involving the $\cG(K)$-symmetry. The general growth assumptions of $g$ and the main results concerning \eqref{eq} are presented in Section \ref{sec:main} and their proofs are contained in Section \ref{sec:proofs}.
Finally, in Section \ref{sec:curlcurl}, we solve the curl-curl problem \eqref{e-normcurl} by means of \eqref{eq}.

\section{Functional setting}\label{sec:setting}

We work in the usual Sobolev space $H^m (\R^N)$ endowed with the norm, equivalent to the standard one,
$$
\| u\|_{H^m}^2 := |\nabla^m u|_2^2 + |u|_2^2,
$$
where $|\cdot|_k$ stands for the usual $L^k$-norm and $\nabla^m$ is given by \eqref{def:nabla-m}. Let us recall the critical Sobolev exponent $2^* := \frac{2N}{N-2m}$ for $N > 2m$ and $2^* := +\infty$ if $N=2m$. Then the following embeddings
$$
\begin{array}{ll}
H^{m} (\R^N) \subset L^t (\R^N), & \quad t \in [2,2^*)\\
H^{m} (\R^N) \subset L^{2^*}(\R^N), & \quad N > 2m
\end{array}
$$
are continuous and, concerning the case $N = 2m$, we have the following Moser--Trudinger-like inequality (see \cite[Theorem 1.1]{Ruf}, \cite[Theorem 1.3]{LamLu}):
\begin{equation}\label{eq:MT}
\sup \left\{\int_{\R^N} \left(\exp\left(\frac{N(2\pi)^N}{\omega_{N-1}} u^2\right) - 1\right) \, dx : \|u\|_{H^{N/2}} \le 1\right\} < +\infty.
\end{equation}

As a consequence, the following property holds true (cf. \cite[Corollary 2.4]{dPS}).
\begin{Lem}\label{lem:N=2s}
Let $\sigma\ge 2$, $M>0$, and $\alpha>0$ such that $\alpha M^2 < \frac{N(2\pi)^N}{\omega_{N-1}}$. Then there exists $C>0$ such that for every $t\in \left(1,N(2\pi)^N/(\omega_{N-1} \alpha M^2)\right]$ and $u\in H^{ N/2}(\R^N)$ with $\|u\|_{H^{N/2}}\le M$, 
\[
\int_{\R^N} |u|^\sigma \left(e^{\alpha u^2}-1\right)\,dx\le C|u|_\frac{\sigma t}{t-1}^\sigma.
\]
\end{Lem}
\begin{proof}
First observe that, if $s\ge0$ and $\beta\ge1$,
\begin{equation}\label{ineqe1}
(e^s-1)^\beta\leq e^{s\beta}-1.
\end{equation}
Let $u\in H^{N/2}(\R^N)$. From H\"older's inequality and \eqref{ineqe1} we have that, for every $t>1$,
\begin{align*}
\int_{\R^N} |u|^\sigma \left(\exp(\alpha u^2)-1\right)\,dx
& \le
|u|_{\frac{\sigma t}{t-1}}^\sigma
\left(\int_{\R^N} \left(\exp(\alpha u^2)-1\right)^t\,dx\right)^{1/t}\\
& \le
|u|_{\frac{\sigma t}{t-1}}^\sigma
\left(\int_{\R^N} \left(\exp(\alpha t u^2)-1\right)\,dx\right)^{1/t}.
\end{align*}	
Moreover, if $t\in \left(1, N(2\pi)^N/(\omega_{N-1} \alpha M^2)\right]$ and $\|u\|_{H^{ N/2}}\le M$, letting $v := u / \|u\|_{H^{ N/2}}$, by \eqref{eq:MT}
\begin{align*}
&\quad \int_{\R^N} \left(\exp(\alpha t u^2)-1\right)\,dx
= \int_{\R^N} \left(\exp\left(\alpha t\|u\|^2 v^2\right) - 1\right)\,dx\\
& \le
\int_{\R^N} \left(\exp\left(\alpha t M^2 v^2\right)-1\right)\,dx
\le
\int_{\R^N} \left(\exp\left( \frac{N(2\pi)^N}{\omega_{N-1}} v^2\right)-1\right)\,dx
\lesssim 1
\end{align*}
and we conclude.
\end{proof}

Before proceeding, we need the following higher-order version of Lions' lemma (cf. \cite{Lions} in the case $m=1$).

\begin{Lem}\label{lem:LIONS}
If $(u_n) \subset H^m(\rn)$ is bounded and there exists $r>0$ such that
\[
\lim_{n \to +\infty} \sup_{x \in \rn} \int_{B(x,r)} u_n^2 \, dx = 0,
\]
then
\[
\int_{\R^N} | \Psi(u_n)| \, dx \to 0 \ as \ n \to +\infty
\]
for every continuous function $\Psi \colon \R \rightarrow \R$ satisfying
\begin{equation}\label{eq:L1}
\lim_{t \to 0} \frac{\Psi(t)}{t^2} = 0
\end{equation}
and
\begin{equation}\label{eq:L2}
\left\{ \begin{array}{ll}
\displaystyle\lim_{|t| \to +\infty} \frac{\Psi(t)}{t^{2^*}} = 0 &\quad \text{if } N > 2m,\\
\displaystyle\lim_{|t| \to +\infty} \frac{\Psi(t)}{e^{\alpha t^2}} = 0 \text{ for all } \alpha > 0 &\quad \text{if } N = 2m.
\end{array} \right.
\end{equation}
\end{Lem}
\begin{proof}
Let us assume first that $N>2m$ and fix $p \in (2,2^*)$. For every $\eps>0$ there exists $c_\eps>0$ such that
\[
\int_{\rn} |\Psi(u_n)| \, dx \le \eps(|u_n|_2^2 + |u_n|_{2^*}^{2^*}) + c_\eps |u_n|_p^p,
\]
therefore it suffices to prove that $\lim_{n \to +\infty} |u_n|_p = 0$ for a particular choice of $p$. For every $x \in \rn$ there holds
\[
|u_n|_{L^p(B(x,r))} \le |u_n|_{L^2(B(x,r))}^{1-\lambda} |u_n|_{L^{2^*}(B(x,r))}^\lambda \le |u_n|_{L^2(B(x,r))}^{1-\lambda} C \|u_n\|_{H^m(B(x,r))}^\lambda,
\]
where $\lambda = (\frac12 - \frac1p) N/m$ and $C>0$ does not depend on $x$. Choosing $p = 2_*$ we obtain
\[
|u_n|_{L^p(B(x,r))}^p \lesssim |u_n|_{L^2(B(x,r))}^{4m/N} \|u_n\|_{H^m(B(x,r))}^2.
\]
Covering $\rn$ with balls of radius $r$ such that each point is contained in at most $N+1$ of them we get
\[
|u_n|_p^p \lesssim \sup_k \|u_k\|_{H^m}^2 \sup_{x \in \rn} \left(\int_{B(x,r)} u_n^2 \, dx\right)^{2m/N} \to 0 \quad \text{ as } n \to +\infty.
\]
Now let us assume $N=2m$ and fix $\sigma\ge2$, $\alpha>0$, and $t>1$ such that $\alpha t \sup_n\|u_n\|_{H^m}^2 < N(2\pi)^N/\omega_{N-1}$. For every $\eps>0$ there exists $c_\eps>0$ such that
\[
\int_{\rn} |\Psi(u_n)| \, dx \le \eps |u_n|_2^2 + c_\eps \int_{\rn} |u_n|^\sigma \left(e^{\alpha u_n^2} - 1\right) \, dx,
\]
thus, utilizing Lemma \ref{lem:N=2s}, it suffices to prove that $\lim_{n \to +\infty} |u_n|_{\sigma t/(t-1)} = 0$ for a particular choice of $\sigma$ and $t$. In particular, taking $\sigma = 2$ and $t=3$, we can argue as before obtaining
\[
|u_n|_{L^3(B(x,r))}^3 \lesssim |u_n|_{L^2(B(x,r))} \|u_n\|_{H^m(B(x,r))}^2
\]
and conclude as in the previous case.
\end{proof}

Now let $\cO(K)$ denote the group of orthogonal $(K \times K)$-matrices acting on $\mathbb{R}^K$. For $2m \leq K \leq N$, by $\cG(K) := \cO(K) \times \mathrm{id}_{N-K}$ we denote the group acting on the first $K$ components in $\mathbb{R}^N = \mathbb{R}^{K} \times \mathbb{R}^{N-K}$. It is clear that this action is isometric. Recall that we introduced the space $H^{m}_{\cG(K)} (\R^N)$ of $\cG(K)$-invariant functions $u \in H^{m} (\R^N)$. We then have the following $\cG(K)$-symmetric version of Lemma \ref{lem:LIONS}, the proof being omitted because similar to the one for \cite[Corollary 3.2]{NonradMed} (see also \cite[Remark 3.3]{NonradMed}) when $N>K$ and a consequence of the compact embedding $H^{m}_{\cG(N)} (\R^N) \hookrightarrow L^p(\rn)$, $2<p<2^*$, when $N=K$.

\begin{Lem}\label{lem:Lions}
Suppose that $(u_n) \subset H^{m}_{\cG(K)} (\R^N)$ is bounded and for all $R > 0$ the following condition holds true
$$
\lim_{n \to +\infty} \sup_{z \in \R^{N-K}} \int_{B((0,z), R)} |u_n|^2 \, dx = 0.
$$
Then
$$
\int_{\R^N} | \Psi(u_n)| \, dx \to 0 \ as \ n \to +\infty
$$
for every continuous function $\Psi \colon \R \rightarrow \R$ satisfying \eqref{eq:L1}--\eqref{eq:L2}.
\end{Lem}

Now we introduce the space
$$
X^m := \left\{ u \in H^{m} (\R^N) \ : \ |\mu|\int_{\R^N} \frac{u^2}{|y|^{2m}} \, dx < +\infty \right\}
$$
and
$$
X_{\cG(K)}^m := X^m \cap H^{m}_{\cG(K)} (\R^N) = \left\{ u \in H^{m}_{\cG(K)} (\R^N) \ : \ |\mu|\int_{\R^N} \frac{u^2}{|y|^{2m}} \, dx < +\infty \right\}.
$$
In particular, when $\mu = 0$, $X^m = H^{m}(\R^N)$.

Let $\cD^{m,2}(\R^N)$ be the completion of $\cC_0^\infty(\R^N)$ with respect to the norm
$
u \mapsto |\nabla^m u|_2.
$
Then we have the following lemma (cf. \cite[Theorem 2.5]{Herbst}, see also \cite[Subsection 1.1]{MY}).
\begin{Lem}\label{lem:Hardy}
If $K>2m$, then for every $u \in \cD^{m,2}(\R^N)$
\[
\int_{\R^N} \frac{u^2}{|y|^{2m}} \, dx  \leq \left( \frac{\Gamma\left(\frac{K-2m}{4}\right)}{2^m\Gamma\left(\frac{K+2m}{4}\right)} \right)^2 \int_{\R^N} |\nabla^m u|^2 \, dx,
\]
where $\Gamma$ is the Gamma function. Consequently, $X_{\cG(K)}^m = H^m_{\cG(K)}(\R^N)$.
\end{Lem}

The foregoing inequality suggests the following assumption about the real parameter $\mu$:
	\begin{equation}\label{M}
\mu > - \left( \frac{2^m\Gamma\left(\frac{K+2m}{4}\right)}{\Gamma\left(\frac{K-2m}{4}\right)} \right)^2 \mbox{ if } K > 2m \quad \mbox{ or } \quad \mu \ge 0 \mbox{ if } K = 2m.
\end{equation}
Then we may introduce the following norm
$$
\|u\|_\mu^2 := \int_{\R^N} |\nabla^{m} u|^2 \, dx + \int_{\R^N} |u|^2 \, dx + \mu\int_{\R^N} \frac{u^2}{|y|^{2m}} \, dx
$$
on $X^m_{\cG(K)}$.

\subsection{The Poho\v{z}aev identity}\label{sec:Poho}

Since we are working with a singular potential, we need an appropriate Poho\v{z}aev identity. In what follows, $\mathbb{B}^d (R)$ denotes the $d$-dimensional ball centred at $0$ of radius $R > 0$. We also write $\mathbb{S}^{d-1}(R) := \partial \mathbb{B}^d(R)$. Moreover, $\cH^d$ denotes the $d$-dimensional Hausdorff measure.

\begin{Prop}\label{poh:local}
Let $\widetilde g \colon \R \rightarrow \R$ be a continuous function satisfying
\begin{empheq}[]{align}
\label{eq:local1} &|\widetilde g(s)| \lesssim |s| + |s|^{2^*-1} && \hbox{if }N > 2m, \\
\label{eq:local2} &\hbox{for all }q \ge 2\hbox{ and }\alpha>\frac{N(2\pi)^N}{\omega_{N-1}},\hbox{ it holds }\displaystyle |\widetilde g(s)| \lesssim |s| + |s|^{q-1}(e^{\alpha s^2}-1) && \hbox{if }N = 2m.
\end{empheq}
Let $u \in X^m_{\cG(K)}$ be a weak solution to 
$$
(-\Delta)^m u + \frac{\mu}{|y|^{2m}} u = \widetilde g(u).
$$
Then
\begin{equation}\label{e-PohSing}
\int_{\R^N} (N-2m) \left(|\nabla^m u|^2 + \frac{\mu}{|y|^{2m}}u^2\right) - 2N \widetilde G(u) \, dx = 0,
\end{equation}
where $\widetilde G(u) := \int_0^u \widetilde g(t) \, dt$.
\end{Prop}

\begin{proof}
From \cite[Corollary 3.4]{Siemianowski} we know that $u \in H^{2m}_{\mathrm{loc}} \left( \R^N \setminus  ( \{0\} \times \R^{N-K} ) \right)$ if $N > 2m$. When $N = 2m$, we have the same result because $u \in L^p (\R^N)$ for all $p \in [2,+\infty)$ and we can argue in a similar way as \cite[Proofs of Theorems 1.4 and 1.1]{LamLu} to obtain $\widetilde g(u) \in L^p (\R^N)$.\\
For every $n \ge 1$ let $\psi_n \in \cC^1_0(\rn)$ radially symmetric such that $0 \le \psi_n \le 1$, $\psi_n(x) = 1$ for every $|x| \le n$, $\psi_n(x) = 0$ for every $|x| \ge 2n$, and $|x||\nabla \psi_n(x)| \lesssim 1$ for every $x \in \rn$.\\
Observing that, for every $i \in \N$,
\[
\Delta^i (\nabla u \cdot x) = 2i \Delta^i u + \nabla \Delta^i u \cdot x,
\]
the following identities hold true
\begin{align*}
\widetilde g(u) (\nabla u \cdot x) \psi_n  = & \; \div \bigl( \psi_n \widetilde G(u) x \bigr) - N \psi_n \widetilde G(u) - \widetilde G(u) \nabla \psi_n \cdot x,
\end{align*}
\begin{align*}
\frac{u}{|y|^{2m}} (\nabla u \cdot x) \psi_n  = & \; \frac12 \div \left( \frac{u^2}{|y|^{2m}} \psi_n x \right) -\frac12 \frac{u^2}{|y|^{2m}} \nabla \psi_n \cdot x - \frac{N-2m}{2} \frac{u^2}{|y|^{2m}} \psi_n,
\end{align*}
\begin{align*}
\Delta^{2k+1} u (\nabla u \cdot x) \psi_n = & \; \div \biggl[\biggl(\Delta^k(x \cdot \nabla u) \nabla \Delta^k u - \frac{|\nabla \Delta^k u|^2}{2} x - \sum_{j=0}^{k-1} \Delta^{2k-j} u \nabla \Delta^j (\nabla u \cdot x)\\
& + \sum_{j=0}^{k-1} \Delta^j(\nabla u \cdot x) \nabla \Delta^{2k-j}u\biggr) \psi_n\biggr] + \frac{N-4k-2}{2} |\nabla \Delta^k u|^2 \psi_n\\
& - \biggl(\Delta^k(\nabla u \cdot x) \nabla \Delta^k u - \frac{|\nabla \Delta^k u|^2}{2} x - \sum_{j=0}^{k-1} \Delta^{2k-j} u \nabla \Delta^j (\nabla u \cdot x)\\
& + \sum_{j=0}^{k-1} \Delta^j(\nabla u \cdot x) \nabla \Delta^{2k-j}u\biggr) \cdot \nabla \psi_n,
\end{align*}
\begin{align*}
\Delta^{2k}u (\nabla u \cdot x) \psi_n = & \; \div\biggl[\biggl(\frac12 (\Delta^ku)^2 x + (\nabla u \cdot x) \nabla \Delta^{2k-1}u + \sum_{j=0}^{k-2} \Delta^{j+1}(\nabla u \cdot x) \nabla \Delta^{2k-j-2}u\\
& - \sum_{j=0}^{k-1} \Delta^{2k-j-1}u \nabla \Delta^j(\nabla u \cdot x)\biggr) \psi_n\biggr] + \frac{4k-N}{2} (\Delta^ku)^2 \psi_n\\
& - \biggl(\frac12 (\Delta^ku)^2 x + (\nabla u \cdot x) \nabla \Delta^{2k-1}u + \sum_{j=0}^{k-2} \Delta^{j+1}(\nabla u \cdot x) \nabla \Delta^{2k-j-2}u\\
& - \sum_{j=0}^{k-1} \Delta^{2k-j-1}u \nabla \Delta^j(\nabla u \cdot x)\biggr) \cdot \nabla \psi_n.
\end{align*}
For every $\eps > 0$ define
\[
C_\eps := \left\{x \in \rn : |y| \ge \eps\right\}.
\]
Multiplying both sides of $(-\Delta)^m u + \frac{\mu}{|y|^{2m}} u = \widetilde g(u)$ by $\psi_n \nabla u \cdot x$, using the identities above, and integrating over $C_\eps$, we obtain
$$
0 = \int_{C_\eps} \left( -(-\Delta)^m u - \frac{\mu}{|y|^{2m}} u + \widetilde g(u) \right) \psi_n \nabla u \cdot x \, dx = A_{n,\eps} + B_{n,\eps},
$$
where
\begin{align*}
A_{n,\eps} &:= \int_{C_\eps} \frac12 |\nabla^m u|^2 \nabla \psi_n \cdot x + \cX \cdot \nabla \psi_n + \frac{N-2m}{2} \psi_n |\nabla^m u|^2 - N \psi_n \widetilde G(u) \\
&\quad - \widetilde G(u) \nabla \psi_n \cdot x + \mu \left( \frac12 \frac{u^2}{|y|^{2m}} \nabla \psi_n \cdot x + \frac{N-2m}{2} \frac{u^2}{|y|^{2m}} \psi_n \right) \, dx, 
\end{align*}
\begin{align*}
B_{n,\eps} &:= \int_{C_\eps} \div \left[ \psi_n \left( - \cX - \frac12 |\nabla^m u|^2 x + \widetilde G(u)x - \frac{\mu}{2} \frac{u^2}{|y|^{2m}} x \right) \right] \, dx\\
&= -\int_{\partial C_\eps} \psi_n \left( - \cX - \frac12 |\nabla^m u|^2 x + \widetilde G(u)x - \frac{\mu}{2} \frac{u^2}{\eps^{2m}} x \right) \cdot \frac{y}{\eps} \, d\cH^{N-1}\\
&= -\int_{\partial C_\eps} \psi_n \left( - \cX \cdot \frac{y}{\eps} - \frac12 |\nabla^m u|^2 \eps + \widetilde G(u)\eps - \frac{\mu}{2} \frac{u^2}{\eps^{2m-1}} \right) \, d\cH^{N-1},
\end{align*}

\begin{equation*}
\cX :=
\begin{cases}
\displaystyle - \Delta^k(\nabla u \cdot x) \nabla \Delta^k u + \sum_{j=0}^{k-1} \Delta^{2k-j} u \nabla \Delta^j (\nabla u \cdot x) - \sum_{j=0}^{k-1} \Delta^j(\nabla u \cdot x) \nabla \Delta^{2k-j}u &\\
\displaystyle \nabla u \cdot x \nabla \Delta^{2k-1} u + \sum_{j=0}^{k-2} \Delta^{j+1}(\nabla u \cdot x) \nabla \Delta^{2k-j-2}u - \sum_{j=0}^{k-1} \Delta^{2k-j-1}u \nabla \Delta^j(\nabla u \cdot x)
\end{cases}
\end{equation*}
if $m=2k+1$ or $m=2k$ respectively, and, with a small abuse of notation, we mean $y = (y,0)$. For a fixed $n \ge 1$, since
$$
0 \le \cH^{N-1}(\partial C_\eps \cap \supp \psi_n) \lesssim \eps^{K-1} n^{N-K} \to 0
$$
as $\eps \to 0^+$, it is clear that
\[
\lim_{\eps \to 0^+} \int_{\partial C_\eps} \psi_n \left| - \cX \cdot \frac{y}{\eps} - \frac12 |\nabla^m u|^2 \eps + \widetilde G(u)\eps \right| \, d\cH^{N-1} = 0,
\]
so we focus on the remaining term. Observe that
$$\int_{\R^K}|y|^{-K}\vp(y)\,dy=\int_{\R^K}\int_{\R^{N-K}}\frac{1}{|y|^{2m}}u^2(y,z)\,dz\,dy<+\infty,$$
where $\vp(y):=|y|^{K-2m}\int_{\R^{N-K}}u^2(y,z)\,dz$, and that $\int_{\R^{N-K}}u^2(\cdot,z)\,dz\in W^{m,1}(\R^K)$.
Since $u$ is $\cG(K)$-symmetric, the trace 
$$
\vp_\eps := \vp|_{\mathbb{S}(\eps)} 
$$
is constant on $\mathbb{S}(\eps)$.
If $\liminf_{\eps\to 0}\vp_\eps>0$, then there exist $\delta,\rho>0$ such that $\vp_\eps\geq \delta$ for $0<\eps\leq \rho$. Then, denoting $\omega_{K-1} := \cH^{K-1}\left(\mathbb{S}^{K-1}(1)\right)$, in view of the coarea formula
\[
\int_{\R^K} |y|^{-K} \vp(y) \, dy \ge \int_{\mathbb{B}^K(\rho)} |y|^{-K} \vp(y) \, dy = \omega_{K-1} \int_0^\rho r^{-1} \vp_r dr \ge \delta \omega_{K-1} \int_0^\rho r^{-1} dr = +\infty
\]
and we get a contradiction. Therefore we find $(\eps_k)\subset (0,+\infty)$ such that $\vp_{\eps_k}\to 0$ as $k\to\infty$. Then
\begin{align*}
0 & \le \frac{1}{\eps_k^{2m-1}} \int_{\partial C_{\eps_k}} \psi_n u^2 \, d\cH^{N-1} \le \frac{1}{\eps_k^{2m-1}}  \int_{\mathbb{S}^{K-1}(\eps_k)}\int_{\R^{N-K}} u^2 \,dz\, d\cH^{K-1}_y\\
& = \eps_k^{1-K}  \int_{\mathbb{S}^{K-1}(\eps_k)}\vp_{\eps_k} d\cH^{K-1}_y = \omega_{K-1} \vp_{\eps_k} \to 0.
\end{align*}
Consequently, we have that $\lim_{k\to+\infty} B_{n,\eps_k} = 0$ for every $n \ge 1$. Moreover, from the dominated convergence theorem,
\begin{align*}
\lim_{k\to+\infty} A_{n,\eps_k} &= \lim_{k\to+\infty} \int_{C_{\eps_k}} \frac12 |\nabla^m u|^2 \nabla \psi_n \cdot x + \cX \cdot \nabla \psi_n + \frac{N-2m}{2} \psi_n |\nabla^m u|^2 - N \psi_n \widetilde G(u) \\
&\quad - \widetilde G(u) \nabla \psi_n \cdot x + \mu \left( \frac12 \frac{u^2}{|y|^{2m}} \nabla \psi_n \cdot x + \frac{N-2m}{2} \frac{u^2}{|y|^{2m}} \psi_n \right) \, dx\\
&= \int_{\rn} \frac12 |\nabla^m u|^2 \nabla \psi_n \cdot x + \cX \cdot \nabla \psi_n + \frac{N-2m}{2} \psi_n |\nabla^m u|^2 - N \psi_n \widetilde G(u) \\
&\quad - \widetilde G(u) \nabla \psi_n \cdot x + \mu \left( \frac12 \frac{u^2}{|y|^{2m}} \nabla \psi_n \cdot x + \frac{N-2m}{2} \frac{u^2}{|y|^{2m}} \psi_n \right) \, dx
\end{align*}
and so, from the properties of $\psi_n$ and again the dominated convergence theorem,
\[
0 = \lim_{n\to+\infty} \lim_{k\to+\infty} A_{n,\eps_k} + B_{n,\eps_k} = \int_{\R^N} (N-2m) \left(|\nabla^m u|^2 + \frac{\mu}{|y|^{2m}}u^2\right) - 2N \widetilde G(u) \, dx. \qedhere
\]
\end{proof}

\begin{Rem}\label{R:Smets}
Unless $N=K$, and even when $m=1$, we cannot argue in the proof above as in \cite[Proposition 2.2]{Smets}, i.e., with a sequence cut-off functions $(\phi_n)$ identically zero on the sets $\mathbb{B}^K(\frac1n) \times \R^{N-K}$. The reason is that we need the condition $|x||\nabla \phi_n(x)| \lesssim 1$, but due to the increasing slope in neighbourhoods of the aforementioned sets, we can only get $|x \cdot \nabla \phi_n(x)| \lesssim 1$ unless every $\phi_n$ is radially symmetric, which we need to rule out if $N>K$. Additionally, if $m\ge2$, the condition equivalent to $|x||\nabla \phi_n(x)| \lesssim 1$ involves higher derivatives of $\phi_n$ and so cannot be satisfied for small $x$, even when $N=K$. Similarly, in the cylindrical setting, we cannot argue as in the proof of \cite[Proposition 2.3]{LLT}.\\
Finally, having to integrate over the sets $C_\eps$ rather than the whole $\rn$, we need to express the foregoing identities for $\Delta^mu (\nabla u \cdot x)$ pointwise and not after integration over $\rn$; in particular, we need the terms
\[
\biggl(\frac12 (\Delta^ku)^2 x + (\nabla u \cdot x) \nabla \Delta^{2k-1}u + \sum_{j=0}^{k-2} \Delta^{j+1}(\nabla u \cdot x) \nabla \Delta^{2k-j-2}u - \sum_{j=0}^{k-1} \Delta^{2k-j-1}u \nabla \Delta^j(\nabla u \cdot x)\biggr) \psi_n,
\]
and
\[
\biggl(\Delta^k(x \cdot \nabla u) \nabla \Delta^k u - \frac{|\nabla \Delta^k u|^2}{2} x - \sum_{j=0}^{k-1} \Delta^{2k-j} u \nabla \Delta^j (\nabla u \cdot x) + \sum_{j=0}^{k-1} \Delta^j(\nabla u \cdot x) \nabla \Delta^{2k-j}u\biggr) \psi_n.
\]
To the best of our knowledge, these pointwise identities were proved only recently \cite{MedSie} when $m=2$ and $\mu=0$, and are completely new when $m\ge3$.
\end{Rem}
Now observe that 
from Proposition \ref{poh:local} we know that every solution $u\in X_{\cG(K)}^m$ to \eqref{eq} satisfies the following Poho\v{z}aev identity
$$
(N-2m)\int_{\R^N} |\nabla^{m} u| \, dx + (N-2m)\int_{\R^N} \frac{\mu}{|y|^{2m}} u^2 \, dx - 2N \int_{\R^N} G(u) - \frac{\lambda}{2} u^2 \, dx = 0
$$
provided that $g$ satisfies \eqref{eq:local1} and \eqref{eq:local2}.
Moreover, every solution to \eqref{eq} satisfies also the Nehari identity
$$
\int_{\R^N} |\nabla^{m} u|^2 \, dx + \int_{\R^N} \frac{\mu}{|y|^{2m}} u^2 \, dx+ \lambda \int_{\R^N} |u|^2 \, dx - \int_{\R^N} g(u)u \, dx = 0.
$$
In particular, every nontrivial solution from $X_{\cG(K)}^m$ lies in $\cM$, where
$$
\cM := \left\{ u \in X_{\cG(K)}^m \setminus \{0\} \ : \ M(u) = 0 \right\}
$$
and
$$
M(u) := \int_{\R^N} |\nabla^{m} u|^2 \, dx + \int_{\R^N} \frac{\mu}{|y|^{2m}} u^2 \, dx - \frac{N}{2m} \int_{\R^N} H(u) \, dx.
$$

\section{Normalized solutions to \eqref{eq}}\label{sec:main}

We recall the Gagliardo--Nirenberg inequality, i.e., there exists $C_{N,p}>0$ such that
\begin{equation}\label{eq:GN}
|u|_p \leq C_{N,p}  |\nabla^{m} u|^\delta_2  |u|_2^{1-\delta} \quad \text{for every } u \in H^{m}(\R^N),
\end{equation}
where $\delta := \frac{N}{m} \left( \frac12 - \frac{1}{p}\right)$ and $C_{N,p}$ is optimal.

Recall that
$$
\cS := \left\{ u \in L^2(\R^N)\ : \ \int_{\R^N} |u|^2 \, dx = \rho \right\}
$$
and
$$
\cD := \left\{ u \in  L^2(\R^N)\ : \ \int_{\R^N} |u|^2 \, dx \leq \rho \right\}.
$$

We introduce the $L^2$-critical exponent in $H^{m}(\R^N)$, i.e., $2_{*} := 2+\frac{4m}{N} < 2^*$. Given functions $f_1,f_2\colon\R\to\R$, we recall the following notation \cite{BM}: 
$f_1(s)\preceq f_2(s)$ for $s\in\R\setminus\{0\}$ provided that 
$f_1(s)\leq f_2(s)$ for all $s\in\R$ and for any $\gamma>0$ there is $|s|<\gamma$ such that $f_1(s)<f_2(s)$. Then we have the following integral charachterization of $\preceq$. The proof is omitted as it is similar as the proof of \cite[Lemma 2.1]{BM} in the case $m=1$.
\begin{Lem}
Suppose that $f_1, f_2 \colon \R\to\R$ are continuous, $f_1(s) \leq f_2(s)$ for every $s \in \R$, and for $j \in \{1,2\}$ there holds
\begin{align*}
&|f_j(s)| \lesssim s^2 + |s|^{2^*} && \hbox{if }N > 2m, \\
&\hbox{for all }q \ge 2\hbox{ and }\alpha>\frac{N(2\pi)^N}{\omega_{N-1}},\hbox{ it holds }\displaystyle |f_j(s)| \lesssim s^2 + |s|^{q}(e^{\alpha s^2}-1) && \hbox{if }N = 2m.
\end{align*}
Then $f_1(s) \preceq f_2(s)$ for $s \in \R \setminus \{0\}$ if and only if 
$$
\int_{\R^N} f_1(u) \, dx < \int_{\R^N} f_2(u) \, dx
$$
for every function $u \in H^m (\R^N)$.
\end{Lem}

The assumptions about $g$ are the following.
\begin{itemize}
	\item[(A0)] $g$ and $h:=H'$ are continuous and
	$$
	\begin{array}{ll}
	|g(s)|+|h(s)|\lesssim |s|+|s|^{2^*-1} & \text{if } N > 2m\\
	\text{for every } q \ge 2 \text{ and } \alpha > \frac{N(2\pi)^N}{\omega_{N-1}}, |g(s)|+|h(s)| \lesssim |s| + |s|^{q-1}(e^{\alpha s^2} - 1) & \text{if } N = 2m
	\end{array}
	$$
	for every $s \in \R$.
	\item[(A1)] $\eta:=\limsup_{|s|\to 0} H(s)/|s|^{2_{*}}<+\infty$.
	\item[(A2)] $\lim_{|s|\to+\infty} G(s)/|s|^{2_{*}}=+\infty$.
	\item[(A3)] $\begin{cases}\lim_{|s|\to+\infty} g(s)/|s|^{2^*-1}= 0, \quad N > 2m,\\ \lim_{|s|\to+\infty} g(s)/e^{\alpha s^2} = 0 \mbox{ for any } \alpha > 0, \quad N=2m. \end{cases}$
	\item[(A4)] $2_{*} H(s) \preceq  h(s)s$ for $s\in \R\setminus\{0\}$. 
	\item[(A5)] $0 \preceq \frac{4m}{N}G(s) \le H(s) \preceq (2^*-2)G(s)$ for $s \in\R\setminus\{0\}$, with the rightmost inequality holding if $N >2m$.
\end{itemize}

Observe that $0 \preceq G$ follows from the other two inequalities in (A5) when $N >2m$. Observe also that (A2) and the middle inequality in (A5) imply that there is $\zeta_0 \neq 0$ such that $H(\zeta_0) > 0$. Moreover, $\cM$ is a $\cC^1$-manifold in $X_{\cG(K)}^m$ with $\codim_{X_{\cG(K)}^m} \cM = 1$, see Lemma \ref{lem:manifold} below. It is easy to see that \eqref{example} satisfies all above assumptions; more examples are provided in \cite{BM,JeanjeanLuNorm}.

To simplify the notations, let
\begin{equation*}
C_{*} := C_{N,2_{*}},
\end{equation*}
where $C_{N,2_{*}}$ is defined in \eqref{eq:GN}.

\begin{Th}\label{th:main}
Suppose that \eqref{M}, (A0)--(A5) are satisfied and
\begin{equation}\label{eq:Hstrict}
\frac{N}{2m} \eta C_{*}^{2_{*}} \rho^{\frac{2m}{N}} \tau^{-1} < 1
\end{equation}
holds, $\tau$ being defined in \eqref{eq:tau}. Then there exist $\lambda>0$ and $u \in \cS \cap \cM$ such that
$$
J(u) = \inf_{\cD \cap \cM} J = \inf_{\cS \cap \cM} J > 0
$$
and $u \in \cS \cap \cM$ is a solution  of \eqref{eq} in $X_{\cG(K)}^m$. If, in addition, $g$ is odd and $m=1$, then $u$ is nonnegative.
\end{Th}

Observe that, when $\mu=0$, we do not need to consider $\cG(K)$-invariant functions and then we can find a normalized ground state solution $u_0\in\cM_0\cap\cS$ under the assumptions of Theorem \ref{th:main}. More precisely, we obtain the following generalization of \cite[Theorem 1.1]{BM}.

\begin{Th}\label{th:generalization}
Suppose that (A0)--(A5) are satisfied and
\begin{equation}
\frac{N}{2m} \eta C_{*}^{2_{*}} \rho^{\frac{2m}{N}} < 1
\end{equation}
holds. Then there exist $\lambda>0$ and $u \in \cS \cap \cM_0$ such that
$$
J_0(u) = \inf_{\cD \cap \cM_0} J_0 = \inf_{\cS \cap \cM_0} J_0 > 0
$$
and $u \in \cS \cap \cM_0$ is a normalized ground state solution of \eqref{eq} with $\mu = 0$ in $H^m (\R^N)$, where $J_0$ and $\cM_0$ are given by \eqref{intro:eq1} and \eqref{intro:M0} respectively. If, in addition, $g$ is odd and $m=1$, then $u$ is positive and radially symmetric.
\end{Th}

Below we justify the use of symmetry $\cG(K)$, i.e., there are no normalized ground state solutions provided that $\mu>0$ and $K > 2m$.
	
\begin{Th}\label{thm:nonexistence} 
Assume that $N\geq K >2m$, $\mu>0$, and suppose that \eqref{M}, (A0)--(A5), and \eqref{eq:Hstrict} are satisfied. Then the following problems admit no minimizers 
\begin{align*}
	&\inf\left\{J(u)\ :\ u\in H^{m}(\R^N)\setminus\{0\}, \int_{\R^N} |\nabla^{m} u|^2 + \frac{\mu}{|y|^{2m}} u^2 \, dx=\frac{N}{2m} \int_{\R^N} H(u) \, dx,\hbox{and }u\in\cD \right\},\\
	&\inf\left\{J(u)\ :\ u\in H^{m}(\R^N)\setminus\{0\}, \int_{\R^N} |\nabla^{m} u|^2 + \frac{\mu}{|y|^{2m}} u^2 \, dx=\frac{N}{2m} \int_{\R^N} H(u) \, dx,\hbox{and }u\in\cS \right\}.
\end{align*}
In particular, \eqref{eq} has no normalized ground state solutions in $H^{m}(\R^N)$.
\end{Th}

\section{The proof}\label{sec:proofs}

Observe that if $u\in X_{\cG(K)}^m\setminus\{0\}$ and $N \ge 2m$, then by (A5), $0 \preceq H(t)$ for $t\in\R \setminus\{0\}$ and $\int_{\R^N}H(u)\,dx>0$. Moreover, $u(r(u)\cdot)\in\cM$ for
\begin{equation}\label{def:r}
r(u) :=  \left( \frac{\frac{N}{2m} \int_{\R^N} H(u) \, dx}{ \int_{\R^N} |\nabla^{m} u|^2 + \frac{\mu}{|y|^{2m}} u^2 \, dx} \right)^{\frac{1}{2}},
\end{equation}
hence $\cM$ is nonempty.

To simplify the notations, we define
$$
[u]_\mu := \left(  \int_{\R^N} |\nabla^{m} u|^2 + \frac{\mu}{|y|^{2m}} u^2 \, dx \right)^{1/2} \geq \tau^{1/2} \left( \int_{\R^N} |\nabla^{m} u|^2 \, dx \right)^{1/2} =: \tau^{1/2} [u],
$$
where
\begin{equation}\label{eq:tau}
\tau :=
\begin{cases}
\displaystyle 1 + \left( \frac{2^m\Gamma\left(\frac{K+2m}{4}\right)}{\Gamma\left(\frac{K-2m}{4}\right)} \right)^2 \mu  & \text{if } \mu < 0,\\
1 & \text{if } \mu \ge 0.
\end{cases}
\end{equation}
Note that $\tau$ is well-defined because $\mu < 0$ implies $K > 2m$, see \eqref{M}.

First, we will show that $\cM$ is a differentiable manifold.
\begin{Lem}\label{lem:manifold}
Suppose that \eqref{M} and (A0) hold and that there exists $\zeta_0 \ne 0$ such that $H(\zeta_0) > 0$. Then $\cM$ is a $\cC^1$-manifold in $X_{\cG(K)}^m$ of codimension 1.
\end{Lem}
\begin{proof}
Suppose that $u \in \cM$. Then $[u]_\mu^2 = \frac{N}{2m} \int_{\R^N} H(u) \, dx$. Suppose that $M'(u) = 0$, i.e., $u$ is a weak solution to 
$$
(-\Delta)^{m} u + \frac{\mu}{|y|^{2m}} u = \frac{N}{2m} h(u).
$$
From the Poho\v{z}aev identity (Proposition \ref{poh:local}), we get that $(N-2m) [u]_\mu^2 = 2N \int_{\R^N} \frac{N}{2m} H(u) \, dx$. Thus 
$$
(N-2m) [u]_\mu^2 = 2N [u]_\mu^2
$$
and $[u]_\mu = 0$. Thus $M'(u) \neq 0$ for all $u \in \cM$ and the proof is complete.
\end{proof}

\begin{Lem}\label{lem1}
$$
\inf_{u \in \cD \cap \cM} [u]_\mu > 0.
$$
\end{Lem}

\begin{proof}
Let us start with the case $N > 2m$.
In view of (A1), (A3), and (A5) for any $\eps>0$ there is $c_\eps>0$ such that
\begin{equation*}
	H(s) \leq (\eps+\eta) |s|^{2_{*}} + c_\varepsilon |s|^{2^*}
\end{equation*}
for any $s\in\R$. From \eqref{eq:GN}
	\begin{align*}
	|u|_{2_{*}} \leq C_{*} \left( \int_{\R^N} |\nabla^{m} u|^2 \, dx \right)^{N/(4m+2N)} |u|_2^{2m/(2m+N)} \leq C_{*} \rho^{m/(2m+N)} [u]^{N/(2m+N)}.
	\end{align*}
	Since $u \in \cD \cap \cM$, we get (note that $c_\eps$ varies from one line to another)
	\begin{align*}
	[u]_\mu^2 &= \frac{N}{2m}   \int_{\R^N} H(u) \, dx \leq \frac{N}{2m} \left( (\eta + \eps) |u|_{2_{*}}^{2_{*}} + c_\varepsilon |u|_{2^*}^{2^*} \right)\\
	&\leq \frac{N}{2m} \left( (\eta + \eps) C_{*}^{2_{*}} \rho^{2m/N} [u]^2 + c_\eps [u]^{2^*} \right) \\
	&\leq \frac{N}{2m} \left( (\eta + \eps) C_{*}^{2_{*}} \rho^{2m/N} \tau^{-1} [u]_\mu^2 + c_\eps [u]_\mu^{2^*} \right) \\
	&\leq \frac{N}{2m} (\eta + \eps) C_{*}^{2_{*}} \rho^{2m/N} \tau^{-1} [u]_\mu^2 + c_\eps [u]_\mu^{2^*}.
	\end{align*}
	Taking $\eps$ sufficiently small, we obtain that
	$[u]_\mu^2$ is bounded away from $0$ on $\cD\cap\cM$ provided that \eqref{eq:Hstrict} holds.

Now let us consider the case $N=2m$. Observe that
\[
[u]_\mu^2 \ge \tau \left(\|u\|_{H^{N/2}}^2 - |u|_2^2\right) \ge \tau \left(\|u\|_{H^{N/2}}^2 - \rho\right),
\]
therefore it is enough to prove that
\[
\inf \left\{[u]_\mu : u \in \cD \cap \cM \text{ and } \|u\|_{H^{N/2}}^2 \le 2\rho\right\} > 0.
\]
Let us take $\displaystyle \alpha \in \left(0, \frac{N(2\pi)^N}{2 \rho \omega_{N-1}}\right)$. Note that in this case $2_* = 4$. In view of (A1) and (A3), for every $\eps > 0$ there exists $c_\eps > 0$ such that
\[
H(u) \le (\eta + \eps) u^4 + c_\eps u^4 \left(\exp(\alpha u^2) - 1\right).
\]
Let us take $\displaystyle t := \frac{N(2\pi)^N}{2\rho \alpha \omega_{N-1}}$. Then, from Lemma \ref{lem:N=2s} and arguing as before,
\begin{align*}
[u]_\mu^2 &=   \int_{\R^{2m}} H(u) \, dx \leq  (\eta + \eps) |u|_{4}^{4} + c_\varepsilon \int_{\R^{2m}} u^4 \left(\exp(\alpha u^2) - 1\right) \, dx \\
& \le  (\eta + \eps) C_{*}^{4} \rho \tau^{-1} [u]_\mu^2 + c_\eps |u|_{4t/(t-1)}^4\\
& \le (\eta + \eps) C_{*}^{4} \rho \tau^{-1} [u]_\mu^2 + c_\eps [u]_\mu^{2+2/t}
\end{align*}
and we conclude similarly.
\end{proof}

We note the following, easy to check, scaling properties.

\begin{Lem}\label{lem:scaling}
Fix $s > 0$ and $u \in X_{\cG(K)}^m \setminus \{0\}$. Let $v := s^{N/2} u(s \cdot)$. Then
\begin{enumerate}
\item[(1)] $|v|_{2} = |u|_{2}$, 
\item[(2)] $[v]_\mu^2 = s^{2m} [u]_\mu^2$.
\end{enumerate}
\end{Lem}

Let $u\in X_{\cG(K)}^m \setminus\{0\}$ and define 
\begin{equation}\label{vp_u}
\vp(s) := \vp_u(s) :=J(s^{\frac{N}{2}} u(s \cdot)),\quad s\in (0,+\infty).
\end{equation} 

\begin{Lem}\label{lem:phi}
For every $u \in X_{\cG(K)}^m$ such that
\begin{equation*}
\eta\frac{N}{2m} < \frac{[u]_\mu^2}{|u|_{2_{*}}^{2_{*}}}
\end{equation*}
there exists $s_0 \in (0,+\infty)$ such that $\vp(s_0) = \max\vp$, $s_0^{\frac{N}{2}} u(s_0 \cdot)\in\cM$, and $\vp(s_0) > \vp(s)$ for every $s\in (0,+\infty) \setminus \{s_0\}$. Moreover, $s_0=1$ if and only if $u \in \cM$.
\end{Lem}
\begin{proof}
Fix $u \in X_{\cG(K)}^m\setminus\{0\}$. Observe that, from Lemma \ref{lem:scaling}, (A1), and (A5),
$$
\varphi(s) = \frac{s^{2m}}{2}  [u]_\mu^2  - \int_{\R^N} \frac{G(s^{N/2}u)}{s^N} \, dx  \to 0
$$
as $s \to 0^+$. Moreover, from (A2) there follows that
$$
\frac{\varphi(s)}{s^{2m}} = \frac12  [u]_\mu^2 - \int_{\R^N} \frac{G(s^{N/2}u)}{ \left(s^{N/2} \right)^{2_{*}} } \, dx \to -\infty
$$
as $s \to +\infty$. In addition, if $N > 2m$, from (A1), (A3), and (A5), for every $\eps>0$ there exists $c_\eps>0$ such that
\[
G(u) \le \frac{N}{4m} H(u) \le \frac{N}{4m} (\eps+\eta)|u|^{2_{*}}+c_\eps|u|^{2^*},
\]
therefore,
\[
\vp(s) \geq s^{2m} \left(\frac12 [u]_\mu^2 - \frac{N}{4m}(\eta + \eps) |u|_{2_{*}}^{2_{*}}\right) - c_\eps s^{2^* m} |u|_{2^*}^{2^*} > 0
\]
for sufficiently small $\eps$ and $s$. If $N = 2m$, from Lemma \ref{lem:N=2s} and arguing as above, we obtain
\[
G(u) \le \frac{1}{2} H(u) \le \frac{1}{2} (\eps+\eta)|u|^{2_{*}} + c_\eps |u|^\sigma\left(\exp(\alpha u^2) - 1\right),
\]
where $\sigma > 4$ is fixed and $\displaystyle \alpha \in \left(0, \frac{N(2\pi)^N}{\|u\|_{H^{N/2}}^2 \omega_{N-1} }\right)$. Therefore, taking $\displaystyle t := \frac{N(2\pi)^N}{\alpha \|u\|_{H^{N/2}}^2 \omega_{N-1} }$ and $s \le 1$,
\begin{align*}
\vp(s) & \ge s^{2m} \left(\frac12 [u]_\mu^2 - (\eta + \eps) |u|_{2_{*}}^{2_{*}}\right) - c_\eps s^{m(\sigma-2)} \int_{\R^{2m}} |u|^\sigma \left(\exp(s^{2m} \alpha u^2) - 1\right) \, dx\\
& \ge s^{2m} \left(\frac12 [u]_\mu^2 - (\eta + \eps) |u|_{2_{*}}^{2_{*}}\right) - c_\eps s^{m(\sigma-2)} |u|_\frac{\sigma t}{t-1}^\sigma > 0
\end{align*}
for sufficiently small $\eps$ and $s$. Hence $\varphi$ has a maximum at some $s_0 > 0$. In particular $\varphi'(s_0) = 0$, so that
$$0=\vp'(s_0)= ms_0^{2m-1}  \left( [u]_\mu^2 -\frac{N}{2m}\int_{\R^N}H(s_0^{N/2} u)(s_0^{N/2})^{-2_{*}}\,dx\right)$$
and 
$$[v]_\mu^2 = \frac{N}{2m} \int_{\R^N} H\left( v \right) \, dx,$$ 
where $v := s_0^{\frac{N}{2}} u(s_0 \cdot)$. Hence $s_0^{\frac{N}{2}} u(s_0 \cdot) \in \cM$. From (A4) there follows that the function
$$
(0, +\infty) \ni s \mapsto \int_{\R^N} H(s^{N/2} u)(s^{N/2})^{-2_{*}} \, dx \in \R
$$
is increasing, hence $s_0$ is unique. If, in addition, $u \in \cM$, then
$$
[u]_\mu^2 = \frac{N}{2m} \int_{\R^N} H(u) \, dx,
$$
so $\varphi'(1) = 0$ and $s_0=1$.
\end{proof}

\begin{Lem}\label{lemCoercive}
$J$ is coercive on $\cD \cap \cM$.
\end{Lem}

\begin{proof}
Observe that for $u \in \cD \cap \cM$, taking (A5) into account, we have
$$
J(u) = J(u) - \frac12 M(u) = \frac{N}{4m} \int_{\R^N} H(u) - \frac{4m}{N} G(u) \, dx \geq 0.
$$
Hence $J$ is bounded below on $\cD \cap \cM$.
Now we follow arguments similar to \cite[Lemma 2.3]{BM}, cf. \cite[Lemma 2.5]{JeanjeanLuNorm}.
Suppose that $(u_n) \subset\cD \cap \cM$ is a sequence such that $\|u_n\| \to +\infty$ and $J(u_n)$ is bounded from above. Since $(u_n) \subset \cD$ we see that $[u_n]_\mu^2 \to +\infty$. Set
$$
s_n := \frac{1}{[u_n]_\mu^{1/m}} > 0,\; \hbox{and }
v_n := s_n^{N/2} u_n \left( s_n \cdot \right).
$$
Note that $s_n \to 0^+$ as $n \to +\infty$. Then 
$$
\int_{\R^N} |v_n|^2 \, dx = \int_{\R^N} |u_n|^2 \, dx \leq \rho,
$$
so that $v_n \in \cD$. Moreover,
$$
[v_n]_\mu^2 = s_n^{2m} [u_n]_\mu^2 = s_n^{2m} s_n^{-2m} = 1.
$$
In particular, $(v_n)$ is bounded in $X_{\cG(K)}^m$. We begin by considering the case $N > K$. Suppose that 
$$
\limsup_{n \to +\infty} \left( \sup_{(0,z) \in \R^K \times \R^{N-K}} \int_{B((0,z),1)} |v_n|^2 \, dx \right) > 0.
$$
Then, up to a subsequence, we can find translations $(0,z_n) \subset \R^K \times \R^{N-K}$ such that
$$
v_n (\cdot + (0,z_n)) \weakto v \neq 0 \quad \hbox{in } X_{\cG(K)}^m
$$
and $v_n(x+(0,z_n)) \to v(x)$ for a.e. $x \in \R^N$. Then by (A2)
\begin{align*}
0 \leq \frac{J(u_n)}{[u_n]_\mu^2} &= \frac12 - \int_{\R^N} \frac{G(u_n)}{[u_n]_\mu^2} \, dx = \frac12 - s_n^{N} s_n^{2m} \int_{\R^N} G(u_n(s_n x)) \, dx \\
&= \frac12 - s_n^{N+2m} \int_{\R^N} G( s_n^{-N/2} v_n ) \, dx = \frac12 - s_n^{N+2m} \int_{\R^N} \frac{G(s_n^{-N/2} v_n)}{ \left| s_n^{-N/2} v_n \right|^{2_{*}} } \left| s_n^{-N/2} v_n \right|^{2_{*}} \, dx \\
&= \frac12 -  \int_{\R^N} \frac{G(s_n^{-N/2} v_n)}{ \left| s_n^{-N/2} v_n \right|^{2_{*}} } \left|  v_n \right|^{2_{*}} \, dx \\
&= \frac12 -  \int_{\R^N} \frac{G(s_n^{-N/2} v_n(x+(0,z_n)))}{ \left| s_n^{-N/2} v_n(x+(0,z_n)) \right|^{2_{*}} } \left|  v_n(x+(0,z_n)) \right|^{2_{*}} \, dx  \to -\infty
\end{align*}
and we obtain a contradiction. Hence we may assume that
$$
\sup_{(0,z) \in \R^K \times \R^{N-K}} \int_{B((0,z),1)} |v_n|^2 \, dx \to 0
$$
and from Lemma \ref{lem:Lions} $v_n \to 0$ in $L^{2_{*}} (\R^N)$. Observe that 
$$
u_n = s_n^{-N/2} v_n \left( \frac{\cdot}{s_n} \right) \in \cM.
$$
Hence, from Lemma \ref{lem:phi}, for any $s >0$ there holds
\begin{align*}
J(u_n) = J \left( s_n^{-N/2} v_n \left( \frac{\cdot}{s_n} \right) \right) \geq J \left( s^{N/2} v_n (s \cdot) \right) = \frac{s^{2m}}{2} - s^{-N} \int_{\R^N} G \left( s^{N/2} v_n \right) \, dx
\end{align*}
and $s^{-N} \int_{\R^N} G \left( s^{N/2} v_n \right) \, dx \to 0$ as $n \to +\infty$. Thus we obtain a contradiction by taking sufficiently large $s > 0$. If $N = K$, then from the compact embedding $X_{\cG(K)}^m \hookrightarrow L^{2_{*}}(\R^N)$ we have $v_n \to v$ in $L^{2_{*}}(\R^N)$ and a.e. in $\R^N$ (along a subsequence). If $v \ne 0$, then arguing as above
\[
0 \leq \frac{J(u_n)}{[u_n]_\mu^2} = \frac12 -  \int_{\R^N} \frac{G(s_n^{-N/2} v_n)}{ \left| s_n^{-N/2} v_n \right|^{2_{*}} } \left|  v_n \right|^{2_{*}} \, dx \to -\infty,
\]
thus $v = 0$ and we conclude as before.
\end{proof}

\begin{Lem}\label{lem:c-positive} There holds
	$$
	c:=\inf_{\cD \cap \cM} J > 0.
	$$
\end{Lem}

\begin{proof}
We will show that for $\rho > 0$ satisfying \eqref{eq:Hstrict} there is $\delta > 0$ such that 
\begin{equation}\label{eq:JJ}
\frac{\beta}{2} [u]_\mu^2  \leq J(u)
\end{equation}
for $u \in \cD \cap X_{\cG(K)}^m$ such that $[u]_\mu \leq \delta$, where
\[
\beta = \frac12 - \frac{N}{4m} \eta C_{*}^{2_{*}} \tau^{-1} \rho^{\frac{2}{N}} > 0.
\]
If $N > 2m$, then from \eqref{eq:GN} and (A5) we obtain
\begin{align*}
\int_{\R^N} G(u) \, dx & \le \frac{N}{4m} \int_{\R^N} H(u) \, dx \leq \frac{N}{4m} (\eps + \eta) |u|_{2_{*}}^{2_{*}} + c_\eps |u|_{2^*}^{2^*}\\
& \leq \frac{N}{4m}  (\eps+\eta) C_{*}^{2_{*}}  \rho^{\frac{2m}{N}} [u]^2 + c_\eps [u]^{2^*} \leq \frac{N}{4m} (\eps+\eta) C_{*}^{2_{*}} \tau^{-1} \rho^{\frac{2m}{N}} [u]_\mu^2 + c_\eps [u]_\mu^{2^*} \\
&= \left( \eps C_1 + c_\eps [u]_\mu^{\frac{4m}{N-2m}} + \frac{N}{4m} \eta C_{*}^{2_{*}}  \tau^{-1} \rho^{\frac{2m}{N}} \right) [u]_\mu^2 \\
&= \left( \eps C_1 + c_\eps [u]_\mu^{\frac{4m}{N-2m}} + \frac{1}{2} - \beta \right) [u]_\mu^2
\end{align*}
and we take
$$
\eps := \frac{\beta}{4 C_1} > 0, \quad \delta := \left( \frac{\beta}{4 c_\eps} \right)^{\frac{N-2m}{4m }} > 0.
$$
If $N = 2m$, we can assume $[u]_\mu \le 1$. Then arguing similarly we get
\begin{align*}
\int_{\R^{2m}} G(u) \, dx & \le \frac{1}{2} (\eps + \eta) |u|_{2_{*}}^{2_{*}} + c_\eps \int_{\R^{2m}} u^4 \left(\exp(\alpha u^2) - 1\right) \, dx\\
& \le \left(\eps C_1 + \frac12 - \beta\right)[u]_\mu^2 + c_\eps |u|_\frac{4 t}{t-1}^2 \le \left(\eps C_1 + \frac12 - \beta\right)[u]_\mu^2 + c_\eps [u]_\mu^{2+2/t},
\end{align*}
where $\displaystyle \alpha \in \left(0,\frac{N(2\pi)^N}{(\tau^{-1}+\rho) \omega_{N-1}}\right)$ and $\displaystyle t = \frac{N(2\pi)^N}{\alpha (\tau^{-1}+\rho) \omega_{N-1} }$, and we take $\eps$ as before and
\[
\delta := \min\left\{1,\left( \frac{\beta}{4 c_\eps} \right)^{t/2}\right\}.
\]
We obtain that
$$
\int_{\R^N} G(u) \, dx \leq \left( \frac{\beta}{4} + \frac{\beta}{4} + \frac12 - \beta \right) [u]_\mu^2 = \left( \frac12 - \frac{\beta}{2} \right) [u]_\mu^2.
$$
Hence
$$
J(u) = \frac12 [u]_\mu^2 - \int_{\R^N} G(u) \, dx \geq \frac12 [u]_\mu^2 - \left( \frac12 - \frac{\beta}{2} \right) [u]_\mu^2 = \frac{\beta}{2} [u]_\mu^2.
$$
Fix $u \in \cD \cap \cM$. Then, from Lemma \ref{lem:phi}, for every $s > 0$ there holds
$$
J(u) \geq J(s^{N/2} u(s \cdot)).
$$
Choose $s := \frac{\delta^{1/m}}{[u]_\mu^{1/m}} > 0$, where $\delta > 0$ is chosen so that \eqref{eq:JJ} holds, and let $v :=  s^{N/2} u(s \cdot)$. Obviously $|v|_2 = |u|_2$ so that $v \in \cD$. Moreover, $[v]_\mu = \delta$. Then
\[
J(u) \geq J(v) \geq \frac{\beta}{2} [v]_\mu^2 = \frac{\beta}{2} \delta^2 >0.\qedhere
\]
\end{proof}

Before we show that $\inf_{\cD \cap \cM} J$ is attained, inspired by G\'erard \cite{Gerard} we prove the following profile decomposition result applied to $H$ satisfying \eqref{eq:L1}--\eqref{eq:L2}.

\begin{Th}\label{ThGerard}
Suppose that $(u_n)\subset X_{\cG(K)}^m$ is bounded. If  $N>K$, then there are sequences 	$(\tu_i)_{i=0}^{+\infty}\subset X_{\cG(K)}^m$, $(z_n^i)_{i=0}^{+\infty}\subset \R^{N-K}$ for any $n\geq 1$, such that $z_n^0=0$, 	$|z_n^i-z_n^j|\rightarrow +\infty$ as $n\to+\infty$ for $i\neq j$, and passing to a subsequence, the following conditions hold for any $i\geq 0$:
\begin{eqnarray}
\nonumber
&& u_n(\cdot+(0,z_n^i))\weakto \tu_i\; \hbox{ in } X_{\cG(K)}^m \text{ as }n\to+\infty,\\
\label{EqSplit2a}
&& \lim_{n\to+\infty} [u_n]_\mu^2 = \sum_{j=0}^i [ \tu_j ]_\mu^2 +\lim_{n\to+\infty} [ v_n^i ]_\mu^2,
\end{eqnarray}
where $v_n^i:=u_n-\sum_{j=0}^i\tu_j(\cdot-(0,z_n^j))$, and
\begin{eqnarray}
&& \limsup_{n\to+\infty}\int_{\R^N}H(u_n)\, dx= \sum_{j=0}^{+\infty} \int_{\R^N}H(\tu_j)\, dx.	\label{EqSplit3a}
\end{eqnarray}
If $N=K$, then there is $\tu_0\in  X_{\cG(K)}^m$ such that $u_n\weakto \tu_0$ in $X_{\cG(K)}^m$ and \eqref{EqSplit2a}--\eqref{EqSplit3a} hold with $\tu_i=0$ for $i\geq 1$.
\end{Th}

\begin{proof}
The proof of case $N=K$ follows directly from Lemma \ref{lem:Lions}.\\
If $N>K$, then when $\mu = 0$, we proceed as in \cite[Proof of Theorem 1.4]{NonradMed} using Lemma \ref{lem:Lions} instead of the classical version of Lions' lemma. Now, it is sufficient to check that for every $i\ge0$
$$
\lim_{n\to+\infty} \int_{\R^N} \frac{|u_n|^2}{|y|^{2m}} \, dx = \sum_{j=0}^i \int_{\R^N} \frac{|\tu_j|^2}{|y|^{2m}} \, dx + \lim_{n\to+\infty} \int_{\R^N} \frac{|v_n^i|^2}{|y|^{2m}} \, dx.
$$
We proceed by induction. If $i=0$, then since $u_n \weakto \tu_i$ in $X_{\cG(K)}^m$ we have
\begin{align*}
\int_{\R^N} \frac{|\tu_0|^2}{|y|^{2m}} \, dx + \lim_{n\to+\infty} \int_{\R^N} \frac{|u_n - \tu_0|^2}{|y|^{2m}} \, dx & = 2\int_{\R^N} \frac{|\tu_0|^2}{|y|^{2m}} \, dx + \lim_{n\to+\infty} \int_{\R^N} \frac{|u_n|^2 - 2 u_n \tu_0}{|y|^{2m}} \, dx\\
& = \lim_{n\to+\infty} \int_{\R^N} \frac{|u_n|^2}{|y|^{2m}} \, dx.
\end{align*}
Now assume that the claim holds for some $i\ge0$. To simplify notations, for $u \in X_{\cG(K)}^m$ and $z \in \R^{N-K}$ we will write $u(\cdot + z)$ instead of $u(\cdot + (0,z))$. Observe that for every $u \in X_{\cG(K)}^m$ and every $(z_n) \subset \R^{N-K}$ such that $|z_n| \to +\infty$ there holds $u(\cdot + z_n) \weakto 0$ in $X_{\cG(K)}^m$. Then
\begin{align*}
&\quad \sum_{j=0}^{i+1} \int_{\R^N} \frac{|\tu_j|^2}{|y|^{2m}} \, dx + \lim_{n\to+\infty} \int_{\R^N} \frac{|v_n^{i+1}|^2}{|y|^{2m}} \, dx \\
&= \sum_{j=0}^{i+1} \int_{\R^N} \frac{|\tu_j|^2}{|y|^{2m}} \, dx + \lim_{n\to+\infty} \int_{\R^N} \frac{|v_n^i - \tu_{i+1}(\cdot - z_n^{i+1})|^2}{|y|^{2m}} \, dx\\
&= \lim_{n\to+\infty} \int_{\R^N} \frac{|u_n|^2}{|y|^{2m}} \, dx + 2 \left(\int_{\R^N} \frac{|\tu_{i+1}|^2}{|y|^{2m}} \, dx - \lim_{n\to+\infty} \int_{\R^N} \frac{\tu_{i+1}(\cdot - z_n^{i+1})v_n^i}{|y|^{2m}} \, dx\right).
\end{align*}
In addition,
\[
\int_{\R^N} \frac{\tu_{i+1}(\cdot - z_n^{i+1})v_n^i}{|y|^{2m}} \, dx = \int_{\R^N} \frac{\tu_{i+1}(\cdot - z_n^{i+1}) u_n - \tu_{i+1}(\cdot - z_n^{i+1}) \sum_{j=0}^i \tu_j(\cdot - z_n^j)}{|y|^{2m}} \, dx,
\]
with
\[
\int_{\R^N} \frac{\tu_{i+1}(\cdot - z_n^{i+1}) u_n}{|y|^{2m}} \, dx = \int_{\R^N} \frac{\tu_{i+1} u_n(\cdot + z_n^{i+1})}{|y|^{2m}} \, dx \to \int_{\R^N} \frac{|\tu_{i+1}|^2}{|y|^{2m}} \, dx
\]
and
\[
\int_{\R^N} \frac{\tu_{i+1}(\cdot - z_n^{i+1}) \tu_j(\cdot - z_n^j)}{|y|^{2m}} \, dx = \int_{\R^N} \frac{\tu_{i+1}(\cdot + z_n^j - z_n^{i+1}) \tu_j}{|y|^{2m}} \, dx \to 0
\]
for every $j \in \{0,\dots,i\}$.
\end{proof}

\begin{Lem}\label{lem:c_0attained} $c=\inf_{\cD \cap \cM} J$ is attained. If, in addition, $g$ is odd and $m=1$, then $c$ is attained by a nonnegative function in $\cD \cap \cM$.
\end{Lem}

\begin{proof}
Take any sequence $(u_n) \subset \cD \cap \cM$ such that $J(u_n) \to c$ and by Lemma \ref{lemCoercive}, $(u_n)$ is bounded in $X_{\cG(K)}^m$. Note that by (A1), (A3), and (A5), we may apply Theorem \ref{ThGerard} and we find a profile decomposition of $(u_n)$ satisfying \eqref{EqSplit2a} and \eqref{EqSplit3a}.
We show that
$$0< [\tu_i]_\mu^2 \leq \frac{N}{2m}\int_{\R^N} H(\tu_i) \, dx$$
for some $i\geq 0$. Let
$$I:=\{i\geq 0: \tu_i\neq 0\}.$$
In view of Lemma \ref{lem1} and \eqref{EqSplit3a}, $I\neq \emptyset$.
Suppose that  
$$
[\tu_i]_\mu^2 >\frac{N}{2m}\int_{\R^N} H(\tu_i) \, dx
$$
for all $i\in I$.
Then by \eqref{EqSplit2a} and \eqref{EqSplit3a}
\begin{eqnarray*}
\limsup_{n\to+\infty} \frac{N}{2m}\int_{\R^N}H(u_n)\,dx&=&\limsup_{n\to+\infty} [u_n]_\mu^2
\geq \sum_{j=0}^{+\infty} [\tu_j]_\mu^2 \\
&=& \sum_{j\in I} [\tu_j]_\mu^2
> \sum_{j=0}^{+\infty} \frac{N}{2m}\int_{\R^N} H(\tu_j) \, dx \\
&=& \limsup_{n\to+\infty} \frac{N}{2m}\int_{\R^N}H(u_n)\,dx,
\end{eqnarray*}
which is a contradiction. 
Therefore there is $i\in I$ such that $r(\tu_i)\geq 1$ defined as in \eqref{def:r} and
$\tu_i(r(u) \cdot) \in \cM$.  Moreover,
	$$\int_{\R^N}|\tu_i(r(\tu_i)\cdot)|^2\,dx=r(\tu_i)^{-N}\int_{\R^N}|\tu_i|^2\,dx\leq r(\tu_i)^{-N}\rho\leq \rho,$$
hence  $\tu_i(r(\tu_i)\cdot)\in\cD\cap\cM$.
Then passing to a subsequence $u_n(x+y_n^i)\to\tu_i(x)$ for a.e. $x\in\R^N$ and by Fatou's lemma and (A5)
	\begin{eqnarray*}
		\inf_{\cD\cap\cM}J
		&\leq& J(\tu_i(r(\tu_i)\cdot))
		=r(\tu_i)^{-N}\frac{N}{4m}\int_{\R^N}H(\tu_i)-\frac{4m}{N}G(\tu_i)\, dx\\
		&\le& \frac{N}{4m}\int_{\R^N}H(\tu_i)-\frac{4m}{N}G(\tu_i)\, dx\\
		&\leq& \liminf_{n\to+\infty} \frac{N}{4m}\int_{\R^N}H(u_n(\cdot+y_n^i))-\frac{4m}{N}G(u_n(\cdot+y_n^i))\, dx\\
		&=& \liminf_{n\to+\infty} J(u_n)=c=\inf_{\cD\cap\cM}J
	\end{eqnarray*}
Therefore
	$r(\tu_i)=1$, $\tu_i\in\cD\cap\cM$  and $J(\tu_i)=c$.

Suppose that $g$ is odd and $m = 1$. Then $G$ and $H$ are even, so that $G(|u|)=G(u)$ and $H(|u|)=H(u)$ for all $u \in X^1_{\cG(K)}$. We define
$
\tv_i := |\tu_i|
$.
Then $|\tv_i|_2 = |\tu_i|_2$, hence $\tv_i \in \cD$. Moreover,
$$
[\tv_i]_\mu^2 \leq [\tu_i]_\mu^2 = \frac{N}{2} \int_{\R^N} H(\tu_i) \, dx = \frac{N}{2} \int_{\R^N} H(\tv_i) \, dx,
$$
and that $r(\tv_i) \geq 1$, where $r$ is given by \eqref{def:r} and $\tv_i(r(\tv_i) \cdot) \in \cM$. Suppose that $r(\tv_i) > 1$. Then, by (G5)
\begin{align*}
\inf_{\cD \cap \cM} J &\leq J(\tv_i (r(\tv_i) \cdot)) = r(\tv_i)^{-N} \frac{N}{4} \int_{\R^N} H(\tv_i) - \frac{4}{N} G(\tv_i) \, dx \\
&< \frac{N}{4} \int_{\R^N} H(\tv_i) - \frac{4}{N} G(\tv_i) \, dx = \frac{N}{4} \int_{\R^N} H(\tu_i) - \frac{4}{N} G(\tu_i) \, dx = J(\tu_i) = \inf_{\cD \cap \cM} J,
\end{align*}
which is a contradiction. Hence $r(\tv_i) = 1$ and $\tv_i \in \cM$. Obviously $J(\tv_i) = \inf_{\cD \cap \cM} J$ and $\tv_i \geq 0$. 
\end{proof}

\begin{Lem}\label{lem:noM0}
$\left\{u \in \cM : \vp_u''(1) = 0\right\} = \emptyset$, where $\vp_u$ is given by \eqref{vp_u}.
\end{Lem}
\begin{proof}
If by contradiction there exists $u \in \cM$ with $\vp_u''(1) = 0$, then from the identities $M(u) = \vp_u''(1) = 0$ we obtain
\[
\int_{\R^N} 2_* H(u) - h(u)u \, dx = 0,
\]
in contrast with (A4) because $u \ne 0$.
\end{proof}

\begin{proof}[Proof of Theorem \ref{th:main}]
From Lemma \ref{lem:c_0attained}, let $\tu \in \cD \cap \cM$ such that $J(u) = c$. Note that for every $v \in \cS \cap \cM$ the functional $\bigl(\Phi'(v),M'(v)\bigr) \colon X_{\cG(K)}^m \to \R^2$ is surjective, where $\Phi(v) := |v|_2^2$: this can be proved as in \cite[Proof of Lemma 5.2]{Soave} using Lemma \ref{lem:noM0}. Then, from \cite[Proposition A.1]{MS}, there exist Lagrange multipliers $\theta\in\R$ and $\lambda \ge 0$ such that $\tu\in\cD\cap\cM$ solves
$$
(-\Delta)^m \tu + \mu \frac{\tu}{|y|^{2m}}-g(\tu)+\lambda \tu +\theta \left( (-\Delta)^m \tu+ \mu \frac{\tu}{|y|^{2m}}-\frac{N}{4m}h(\tu)\right)=0,
$$
that is,
\begin{equation}\label{eq:mu}
(1+\theta)\left( (-\Delta)^m \tu +  \mu \frac{\tu}{|y|^{2m}} \right)+\lambda\tu =g(\tu)+\frac{N}{4m}\theta h(\tu).
\end{equation}
If $\theta=-1$, then (A4) and (A5) yield
\begin{align*}
0 \leq \lambda \int_{\R^N} \tu^2 \, dx &= \int_{\R^N} g(\tu)\tu+\frac{N}{4m}\theta h(\tu) \tu \, dx = \frac{N}{4m} \int_{\R^N} \frac{4m}{N} g(\tu)\tu- h(\tu) \tu \, dx \\
&< \frac{N}{2m} \int_{\R^N} \frac{4m}{N}  G(\tu)-  H(\tu) \, dx \le 0,
\end{align*}
which is a contradiction. Consequently, $\tu$ satisfies the Nehari and Poho\v{z}aev identities related to equation \eqref{eq:mu}, whence
\[
(1+\theta) [\tu]_\mu^2 = \frac{N}{2m} \int_{\R^N} H(\tu) + \frac{N}{4m} \theta \bigl( h(\tu)\tu - 2 H(\tu) \bigr) \, dx.
\]
This and $\tu \in \cM$ imply
\[
\theta \int_{\R^N} h(\tu)\tu - 2_* H(\tu) \, dx = 0,
\]
hence $\theta = 0$ from (A4). Therefore, from \eqref{eq:mu}, we obtain that $\tu$ is a week solution to the differential equation in \eqref{eq}. Observe that $\lambda = 0$ if $\tu \in \cD \setminus \cS$, thus the proof is concluded once we have proved that $\lambda > 0$. Suppose by contradiction that $\lambda = 0$. Then $\tu$ satisfies
\[
(-\Delta)^m \tu+\frac{\mu}{|y|^{2m}} \tu=  g(\tu)
\]
and so, from the Nehari identity and $\tu \in \cM$, we obtain
\[
2N \int_{\R^N} G(\tu) \, dx = (N-2m) \int_{\R^N} g(\tu) \tu \, dx,
\]
in contrast with (A5). Observe that  if, in addition, $g$ is odd and $m=1$, then the minimizing sequence $(u_n)$ in Lemma \ref{lem:c_0attained} can be replaced by $(|u_n|)$, and finally we get  $\tu$ is nonnegative.
\end{proof}

\begin{proof}[Proof of Theorem \ref{th:generalization}]
We argue as in Lemmas \ref{lem1} -- \ref{lem:c-positive}. If $\mu=0$, then 
Theorem \ref{ThGerard} holds true for a sequence $(u_n)\subset H^m(\R^N)$. Then similarly as in proof of Theorem \ref{th:main} we show the existence of a ground state solution $u \in \cS \cap \cM_0$  in $H^m (\R^N)$ of \eqref{eq} with $\mu = 0$.  Observe that  if, in addition, $g$ is odd and $m=1$, then the minimizing sequence $(u_n)$ in Lemma \ref{lem:c_0attained} can be replaced with $(|u_n|^*)$, where $|u_n|^*$ stands for the Schwarz symmetrization of $|u_n|$. Finally we get a nonnegative and radially symmetric normalized ground state solution. Then, from the maximum principle, $u$ is positive on $\R^N$.
\end{proof}

Note that, arguing as in \cite[Lemma 3.1]{GuoMed} and using Lemma \ref{lem:Hardy}, we can show the following property.

\begin{Lem}\label{lem:MedGuo-polyharm}
Suppose that $N \geq K > 2m$. If $(y_n) \subset \R^K$, $|y_n| \to +\infty$, then for any $u \in H^m(\R^N)$,
$$
\int_{\R^N} \frac{1}{|y|^{2m}} | u(\cdot - y_n)|^2 \, dx \to 0.
$$
\end{Lem}

\begin{altproof}{Theorem \ref{thm:nonexistence}}
Suppose that there is $u_1\in\cM_\mu\cap\cD$ such that $J(u_1)=\inf_{\cM_\mu\cap\cD}J$, where
$$\cM_\mu:=\Big\{u\in H^{m}(\R^N)\setminus\{0\}\ :\ \int_{\R^N} |\nabla^{m} u|^2 \, dx + \int_{\R^N} \frac{\mu}{|y|^{2m}} u^2 \, dx=\frac{N}{2m} \int_{\R^N} H(u) \, dx \Big\}.$$
We take $u_0\neq 0$ as in Theorem \ref{th:generalization}.
Let $u_\theta(y,z):=u_0(y+\theta,z)$ for $\theta\in\R^K$. By  Lemma \ref{lem:phi}, there exists  $s_\theta>0$ such that $s_\theta^{\frac{N}{2}}u_\theta(s_\theta\cdot)\in\cM$. Observe that
$$
\frac{N}{2m}\int_{\R^N}H(s_\theta^{N/2} u_0)(s_\theta^{N/2})^{-2_{*}}\,dx=\frac{N}{2m}\int_{\R^N}H(s_\theta^{N/2} u_\theta)(s_\theta^{N/2})^{-2_{*}}\,dx=[u_\theta]_\mu^2\to \int_{\R^N}|\nabla^{m} u_0|^2\,dx
$$
as $|\theta|\to\infty$, since the integral with the singular potential tends to $0$, see Lemma \ref{lem:MedGuo-polyharm}. By an inspection of proof of Lemma \ref{lem:phi}, note that
$$
(0, +\infty) \ni s \mapsto \xi(s):=\int_{\R^N} H(s^{N/2} u_0)(s^{N/2})^{-2_{*}} \, dx \in (0,+\infty)
$$
is increasing, surjective and $\lim_{s\to 0^+}\xi(s)=0$. Therefore we may assume that $s_{\theta_n}\to s_\infty$ as $n\to\infty$ for some $s_\infty>0$ and $(\theta_n)\subset \R^K$ such that $|\theta_n|\to+\infty$. Observe that
$$\int_{\R^N} \frac{1}{|y|^{2m}} u_{\theta_n}^2 \, dx=\int_{\R^N} \frac{1}{|y-\theta_n|^{2m}} u_{0}^2 \, dx\to 0,$$
and by Lemma \ref{lem:phi} we get
$$\lim_{n\to+\infty}J(s_{\theta_n}^{\frac{N}{2}}u_{\theta_n}(s_{\theta_n}\cdot))=J_0(s_\infty^{\frac{N}{2}}u_0(s_\infty\cdot))\leq J_0(u_0).$$
Observe that $s^{\frac{N}{2}}u_1(s\cdot)\in\cM_0$ for some $s>0$, $s^{\frac{N}{2}}u_1(s\cdot)\in\cD$
and
$$
J(u_1)\geq J(s^{\frac{N}{2}}u_1(s\cdot))>J_0(s^{\frac{N}{2}}u_1(s\cdot))\geq J_0(u_0).$$
Again, by Lemma \ref{lem:phi}, 
$$\lim_{n\to+\infty}J(s_{\theta_n}^{\frac{N}{2}}u_{\theta_n}(s_{\theta_n}\cdot))\geq J(u_1)=\inf_{\cM\cap\cD}J
$$
and we get a contradiction.

We argue similarly about the minimization problem $\inf_{\cM_\mu\cap\cS}$.
\end{altproof}

\section{The curl-curl problem}\label{sec:curlcurl}

We start with a short introduction about the definition of $\nabla \times \nabla \times$ in dimension $N\ge3$. Let us recall that $\nabla \times$ is usually defined in dimension $3$ and that for $\mathbf{U} \in \cC^2 (\R^3; \R^3)$ we have the following identity
\begin{equation}\label{curlN}
\curl \curl \mathbf{U} = \nabla \left( \div \mathbf{U} \right) - \Delta \mathbf{U}.
\end{equation}
Since the right-hand side is well-defined also for vector-valued functions $\mathbf{U} \in \cC^2 (\R^N; \R^N)$, we can treat \eqref{curlN} as the definition of the curl-curl operator in $\R^N$, $N \ge 3$. Similarly, we can propose the weak formulation of this operator. Namely, for any $\mathbf{U}, \mathbf{V} \in \cC^1 (\R^N; \R^N)$ such that $\nabla \mathbf{U}, \nabla \mathbf{V} \in L^2 (\R^{N \times N})$ we define
$$
\int_{\R^N} \langle \curl \mathbf{U}, \curl \mathbf{V} \rangle \, dx := \int_{\R^N} \langle \nabla \mathbf{U}, \nabla \mathbf{V} \rangle - (\div \mathbf{U}) (\div \mathbf{V}) \, dx.
$$

In this section, we are interested in \eqref{e-normcurl} with $N\ge3$, where $\rho>0$ is prescribed and $(\lambda,\UU)$ is the unknown. Looking for classical solutions of the form 
\begin{equation}\label{e-uU}
\mathbf{U}(x) = \frac{u(x)}{r} \left( \begin{array}{c}
-x_2 \\ x_1 \\ 0
\end{array} \right), \quad r = \sqrt{x_1^2+x_2^2}, \quad 0 = 0_{\R^{N-2}},
\end{equation}
we see that $\div \mathbf{U} = 0$ and $u$ satisfies
$$
\left\{ \begin{array}{ll}
-\Delta u + \frac{u}{r^2} + \lambda u = g(u) & \quad \mbox{in } \R^N \\
\int_{\R^N} |u|^2 \, dx = \rho, &
\end{array} \right.
$$
where $f \colon \rn \to \rn$ and $g \colon \R \to \R$ are related by \eqref{e-fgh}, i.e., the problem \eqref{eq} with $m=1$, $\mu=1$, $K = 2$. The same equivalence holds also for weak solutions, see Theorem \ref{T:equiv} and Proposition \ref{P:equiv} below (cf. also \cite[Theorem 2.1]{GMS}, \cite[Theorem 1.1]{Bieg}).  Define $F,\overline{H} \colon \R^N \rightarrow \R$ and $\overline{h} \colon \rn \to \rn$ as $F(\UU) := \int_0^1 f(t\UU) \cdot \UU \, dt$, $\overline H(\UU) := f(\UU) \cdot \UU - 2 F(\UU)$, and $\overline h:= \nabla \overline H$. Observe that $\overline H$ is a radial function, thus $\overline{H}(\UU) = H(u)$ and the pair $(\overline{h},h)$ satisfies \eqref{e-fgh}.

Since we are looking for solutions of the form \eqref{e-uU}, we introduce the following family of functions
$$
\cF := \left\{ \UU \colon \R^N \rightarrow \R^N \ : \ \UU(x) = \frac{u(x)}{r} \left( \begin{array}{c}
-x_2 \\ x_1 \\ 0 
\end{array} \right) \mbox{ for some } \cG(2)\mbox{-invariant } u\colon\R^N\to\R \right\}
$$
and the function space
$$
H_\cF := H^1 (\R^N, \R^N) \cap \cF.
$$
Note that if $\mathbf{U} \in H_\cF$, then $\mathbf{U}$ is ($\SO(2) \times \mathrm{id}$)-equivariant. To utilize the method described in Section \ref{sec:main}, we introduce the $L^2$-sphere and the $L^2$-disc
\begin{align*}
\cD & := \left\{\UU\in L^2(\R^N, \R^N) \ : \ \int_{\rn}|\UU|^2\,dx\le\rho \right\}, \\
\cS & := \left\{\UU\in L^2(\R^N, \R^N) \ : \ \int_{\rn}|\UU|^2\,dx=\rho \right\}.
\end{align*}
Let $E : H_\cF \rightarrow \R$ denote the energy functional associated with \eqref{e-normcurl}, namely,
$$
E (\UU) := \int_{\rn} \frac12 |\nabla\times\UU|^2 - G(\UU) \, dx, \quad \UU \in H_\cF. 
$$
Then, we say that $\UU$ is a normalized solution to \eqref{e-normcurl} if it is a critical point of $E$ restricted to $\cS$. 

We note the following equivalence of weak solutions.

\begin{Th}\label{T:equiv}
Assume $g\colon\R\to\R$ is a continuous function and $|g(u)|\lesssim |u|+|u|^{2^*-1}$ for every $u\in\R$; let also $f$ satisfy \eqref{e-fgh}. Finally, let $u \colon \rn \to \R$ and $\UU \colon \rn \to \rn$ be related by \eqref{e-uU}  for a.e. $x\in\R^N$. Then
$\UU\in H_\cF$ if and only if $u\in X_{\cG(2)}^1$ and, in such a case, $\div \UU = 0$ and
$E(\UU) = J(u)$, where $J$ is given by \eqref{def:J}. Moreover,
$(\lambda,\UU)$ is a solution to \eqref{e-normcurl} 
if and only if $(\lambda,u)$ is a solution to \eqref{eq} with $m=1$, $\mu = 1$ and $K = 2$.
\end{Th}

When $N=3$, this follows from \cite[Theorem 1.1]{Bieg}, while the generalization to the case $N\ge3$ is trivial,  cf. \cite[Section 4.2]{PhD}.

If $\UU \in H_\cF$ is a normalized solution, then since $\div \UU = 0$, it is clear from \eqref{curlN} that $\UU$ is a normalized solution to the vectorial Schr\"odinger equation
$$
-\Delta \UU + \lambda \UU = f(\UU)
$$
and therefore, it satisfies the corresponding Nehari and Poho\v{z}aev identities. Thus, every normalized solution lies in the set
$$
\cQ := \left\{\UU \in H_\cF \setminus \{0\} : \int_{\rn} |\nabla\times\UU|^2 \, dx = \frac{N}2 \int_{\rn} \overline H(\UU) \, dx\right\}.
$$
One can easily compute that $\cQ$ is a differentiable submanifold in $H_\cF$ of class $\cC^1$ and codimension $1$ (cf. Lemma \ref{lem:manifold}). 

Let $\UU \in H_\cF$. Then one has $|\UU|=|u|$, $|\nabla\times\UU|_2=|\nabla\UU|_2=|\nabla u|_2$, $\int_{\rn}F(\UU)\,dx=\int_{\rn}G(u)\,dx$, and $\int_{\rn}f(\UU)\cdot\UU\,dx=\int_{\rn}g(u)u\,dx$, which immediately implies the following property.

\begin{Prop}\label{P:equiv}
Assume $g \colon \R \to \R$ is continuous and $|g(u)|\lesssim |u|+|u|^{2^*-1}$ for every $u\in\R$; let also $f$ satisfy \eqref{e-fgh} and define $u,\UU$ via \eqref{e-uU}. Then $\UU \in \cQ$ if and only if $u \in \cM$, where $\cM$ is given by \eqref{def:M}.
\end{Prop}

The main result of this section reads as follows.

\begin{Th}\label{T:Vec}
If $f$ satisfies \eqref{e-fgh}, $g$ satisfies (A0)--(A5), and \eqref{eq:Hstrict} holds, then there exist $\lambda>0$ and $\UU \in \cQ \cap \cS$ such that $(\lambda,\UU)$ is a solution to \eqref{e-normcurl} and $E(\UU)=\inf_{\cQ \cap \cD} E = \inf_{\cQ \cap \cS} E > 0$.
\end{Th}

\begin{proof}[Proof of Theorem \ref{T:Vec}]
It follows from Theorem \ref{T:equiv}, Theorem \ref{th:main}, and Proposition \ref{P:equiv}.
\end{proof}

\section*{Acknowledgements}
Bartosz Bieganowski and Jarosław Mederski were partly supported by the National Science Centre, Poland, (Grant No. 2017/26/E/ST1/00817).
Jacopo Schino is a member of GNAMPA (INdAM) and was partly supported by the National Science Centre, Poland (Grant No. 2020/37/N/ST1/00795) and the GNAMPA project \textit{Metodi variazionali e topologici per alcune equazioni di Schr\"odinger nonlinari}. We would like to thank the referees for valuable comments and pointing out the recent work \cite{Liu_Zhao}, where the authors obtained further results in case $m=1$ based on our paper.

%


\begin{thebibliography}{99}
\baselineskip 2 mm

\bibitem{Agrawal} G. Agrawal: {\em Nonlinear Fiber Optics}, Academic Press 2019, 6-th Edition.

\bibitem{Akhmediev-etal} N.N. Akhmediev, A. Ankiewicz, J.M. Soto-Crespo: {\em Does the nonlinear Schr\"odinger equation correctly describe beam propagation?}, Opt. Lett. {\bf 18} (1993), 411.


\bibitem{BadBenRol} M. Badiale, V. Benci, S. Rolando: {\em A nonlinear elliptic equation with singular potential and applications to nonlinear field equations}, J. Eur. Math. Soc. (JEMS) {\bf 9} (2007), no. 3, 355--381.


\bibitem{BartschJS2016} T. Bartsch, L. Jeanjean, N. Soave: {\em Normalized solutions for a system of coupled cubic Schr\"odinger equations on $\R^3$}, J. Math. Pures Appl., {\bf 106} (4) (2016), 583--614.

\bibitem{BartschMederski1} T. Bartsch, J. Mederski:
{\em Ground and bound state solutions of semilinear time-harmonic Maxwell equations in a bounded domain}, Arch. Rational Mech. Anal., 215 (1), (2015), 283--306.

\bibitem{BartschMolle} T. Bartsch, R. Molle, M. Rizzi, G. Verzini: \textit{Normalized solutions of mass supercritical Schr\"odinger equations with potential}, Comm. Partial Differential Equations {\bf 46} (2021), no. 9, 1729--1756.

\bibitem{BartschSoaveJFA} T. Bartsch and N. Soave: {\em A natural constraint approach to normalized solutions of nonlinear Schr\"odinger equations and systems}, J. Funct. Anal. {\bf 272} (2017), no. 12, 4998--5037. 

\bibitem{BartschSoaveJFACorr} T. Bartsch and N. Soave: Corrigendum: {\em Correction to: A natural constraint approach to normalized solutions of nonlinear Schr\"odinger equations and systems}, J. Funct. Anal., {\bf 275} (2018), no. 2, 516--521.

\bibitem{BerLions} H. Berestycki, P.L. Lions: {\em Nonlinear scalar field equations. I - existence of a ground state}, Arch. Ration. Mech. Anal. {\bf 82} (1983), no. 4, 313--345.

\bibitem{Bieg} B. Bieganowski: \textit{Solutions to a nonlinear Maxwell equation with two competing nonlinearities in $\mathbb{R}^3$}, Bulletin Polish Acad. Sci. Math. {\bf 69} (2021), no. 1, 37--60.

\bibitem{BM} B. Bieganowski, J. Mederski: \textit{Normalized ground states of the nonlinear Schrödinger equation with at least mass critical growth}, J. Funct. Anal. \textbf{280} (2021), no. 11, Paper No. 108989, 26 pp.


\bibitem{Buryak} A.V. Buryak, P.D. Trapani, D.V. Skryabin, S. Trillo: {\em Optical solitons due to
quadratic nonlinearities: from basic physics to futuristic applications}, Physics Reports {\bf 370} (2002), no. 2, 63--235.

\bibitem{CLY} X. Chang, M. Liu, D. Yan: {\em Normalized ground state solutions of nonlinear Schr\"odinger equations involving exponential critical growth}, J. Geom. Anal. \textbf{33} (2023), no. 3, Paper No. 83, 20 pp.

\bibitem{Ciattoni-etal:2005} A. Ciattoni, B. Crossignani, P. Di Porto, A. Yariv: {\em Perfect optical solitons: spatial Kerr solitons as exact solutions of Maxwell's equations}, J. Opt. Soc. Am. B {\bf 22} (2005), 1384--94.

\bibitem{CSz} M. Clapp, A. Szulkin: \textit{Normalized solutions to a non-variational Schrödinger system}, Topol. Methods Nonlinear Anal. {\bf 59} (2022), no. 2A, 553--568.

\bibitem{dPS} P. d'Avenia, A. Pomponio, J. Schino: \textit{Radial and non-radial multiple solutions to a general mixed dispersion NLS equation}, Nonlinearity {\bf 36} (2023), no. 3, 1743--1775.

\bibitem{Ding} Y. Ding, X. Zhong: {\em  Normalized solution to the Schr\"odinger equation with potential and general nonlinear term: mass super-critical case}, J. Differential Equations {\bf 334} (2022), 194--215.

\bibitem{GMS} M. Gaczkowski, J. Mederski, J. Schino: \textit{Multiple solutions to cylindrically symmetric curl-curl problems and related Schrödinger equations with singular potentials}, SIAM J. Math. Anal. \textbf{55} (2023), no. 5, 4425--4444.

\bibitem{Gazzola} F. Gazzola, H.-C. Grunau, G. Sweers: \textit{Polyharmonic boundary value problems}, Lecture Notes in Mathematics, Springer-Verlag, Berlin (2010).

\bibitem{Gerard} P. G\'erard: {\em Description du d\'efaut de compacit\'e de l'injection de Sobolev}, ESAIM: Control, Optimisation and Calculus of Variations {\bf 3} (1998), 213--233.

\bibitem{GuoMed} Q. Guo, J. Mederski: {\em Ground states of nonlinear Schr\"odinger equations with sum of periodic and inverse-square potentials},
J. Differential Equations {\bf 260} (2016), no. 5, 4180--4202.

\bibitem{Herbst} Ira W. Herbst: \textit{Spectral Theory of the operator $(p^2 + m^2)^{1/2} - Z e^2 / r$}, Comm. Math. Phys. \textbf{53} (1977), no. 3, 285--294.

\bibitem{IK} B. A. Ivanov, A. M. Kosevich: {\em Stable three-dimensional small-amplitude soliton in magnetic materials}, So. J. Low Temp. Phys {\bf 9} (1983), 439--442.

\bibitem{Jeanjean} L. Jeanjean: {\em Existence of solutions with prescribed norm for semilinear elliptic equations}, Nonlinear Anal. \textbf{28} (1997), no. 10, 1633–1659.

\bibitem{JeanjeanLe} L. Jeanjean,  T.-T. Le: {\em Multiple normalized solutions for a Sobolev critical Schr\"odinger equation},
Math. Ann. {\bf 384} (2022), no. 1-2, 101--134.

\bibitem{JeanjeanLuNorm} L. Jeanjean, S.-S. Lu: {\em A mass supercritical problem revisited}, Calc. Var. Partial Differential Equations \textbf{59} (2020), no. 5, Paper No. 174, 43 pp.

\bibitem{LamLu} N. Lam and Gu. Lu: \textit{A new approach to sharp Moser--Trudinger and Adams type inequalities: A rearrangement-free argument}, J. Differential Equations \textbf{255} (2013), no. 3, 298–325.

\bibitem{LLT} G.-D. Li, Y.-Y. Li, C.-L. Tang: \textit{Existence and asymptotic behavior of ground state solutions for Schr\"odinger equations with Hardy potential and Berestycki-Lions type conditions}, J. Differential Equations \textbf{275} (2021), 77--115.

\bibitem{Li_Zou} H. Li, W. Zou: {\em Normalized ground state for the Sobolev critical Schr\"odinger equation involving Hardy term with combined nonlinearities}, Math. Nachr. \textbf{296} (2023), no. 6, 2440--2466.

\bibitem{Liu_Zhao} Y. Liu, L. Zhao: \textit{Normalized solutions for Schr\"odinger equations with 	potentials and general nonlinearities}, Calc Var. Partial Differential Equations \textbf{63} (2024), no. 4, Paper No. 99, 37 pp.

\bibitem{Lions} P.-L. Lions: \textit{The concentration-compactness principle in the calculus of variations. The locally compact case. Part I and II}, Ann. Inst. H. Poincar\'e, Anal. Non Lin\'eare. \textbf{1} (1984); no. 2, 109--145; no. 4, 223--283.

\bibitem{Ma_Chang} Z. Ma, X. Chang: {\em Normalized ground states of nonlinear biharmonic Schr\"odinger equations with Sobolev critical growth and combined nonlinearities}, Appl. Math. Lett. {\bf 135} (2023), Paper No. 108388, 7pp.

\bibitem{Malomed} B. Malomed: {\em Multi-component Bose-Einstein condensates: Theory} in: P.G. Kevrekidis, D.J. Frantzeskakis, R. Carretero-Gonzalez (Eds.): {\em Emergent Nonlinear Phenomena in Bose-Einstein Condensation}, Springer-Verlag, Berlin, 2008, 287--305.

\bibitem{McLeod} J. B. McLeod, C. A. Stuart, W. C. Troy: {\em An exact reduction of Maxwell’s equations, Nonlinear diffusion equations and their equilibrium states}, 3 (Gregynog 1989), 391--405, Progr. Nonlinear Differential Equations Appl., vol. 7, Birkh\"auser Boston, Boston, MA (1992).

\bibitem{NonradMed} J. Mederski: {\em Nonradial solutions for nonlinear scalar field equations}, Nonlinearity {\bf 33} (2020), no. 12, 6349--6380.


\bibitem{MS} J. Mederski, J. Schino: \textit{Least energy solutions to a cooperative system of Schrödinger equations with prescribed $L^2$-bounds: at least $L^2$-critical growth}, Calc. Var. Partial Differential Equations {A\bf 61} (2022), no. 1, Paper No. 10, 31 pp.

\bibitem{MederskiSchinoSurvey} J. Mederski, J. Schino: {\em Nonlinear curl-curl problems in $\mathbb{R}^3$}, Minimax Theory and its Applications {\bf 7} (2022), no. 2, 339--364.

\bibitem{MedSie} J. Mederski, J. Siemianowski: \textit{Biharmonic nonlinear scalar field equations}, Int. Math. Res. Not. \textbf{2023} (2023), Issue 23, 19963--19995.

\bibitem{MY} H. Mitzutani, X. Yao: \textit{Kato smoothing, Strichartz and uniform Sobolev estimates for fractional operators with sharp Hardy potentials}, Comm. Math. Phys. \textbf{388} (2021), no. 1, 581--623.

\bibitem{Molle} R. Molle, G. Riey, G. Verzini: \textit{Normalized solutions to mass supercritical Schr\"odinger equations with negative potential}, Journal of Differential Equations {\bf 333} (2022), 302--331.

\bibitem{Noris} B. Noris, H. Tavares, G.  Verzini: {\em Stable solitary waves with prescribed $L^2$-mass for the cubic Schr\"odinger system with trapping potentials}, Discrete Contin. Dyn. Sys. {\bf 35} (2015), no. 12, 6085--6112.

\bibitem{Phan} T. V. Phan: {\em Blowup for biharmonic Schr\"odinger equation with critical nonlinearity}, Z. Angew. Math. Phys. {\bf 69} (2018), no. 2, Paper No. 31, 11 pp.

\bibitem{PucciSerrin} P. Pucci, J. Serrin: {\em A General Variational Identity}, Indiana Univ. Math. J. {\bf 35} no. 3 (1986), 681--703.

\bibitem{Ruf} B. Ruf: \textit{A Sharp Trudinger-Moser Type Inequality for Unbounded Domains in $\mathbb{R}^2$}, J. Func. Anal. \textbf{219} (2005), no. 2, 340--367.

\bibitem{PhD} J. Schino: \textit{Ground state, bound state, and normalized solutions to semilinear Maxwell and Schr\"odinger equations}, Ph.D. thesis, arXiv:2207.07461.

\bibitem{Schino} J. Schino: \textit{Normalized ground states to a cooperative system of Schrödinger equations with generic $L^2$-subcritical or $L^2$-critical nonlinearity}, Adv. Differential Equations {\bf 27} (2022), no. 7-8, 467--496.

\bibitem{Siemianowski} J. Siemianowski: \textit{Brezis--Kato Type Regularity Results for Higher Order Elliptic Operators}, to appear in Topol. Methods Nonlinear Anal. (2024), arXiv:2202.11408.

\bibitem{Smets} D. Smets: \textit{Nonlinear Schr\"odinger equations with Hardy potential and critical nonlinearities}, Trans. Amer. Math. Soc. {\bf 357} (2005), no. 7, 2909–2938.

\bibitem{Soave} N. Soave: {\em Normalized ground states for the NLS equation with combined nonlinearities}, J. Differential Equations, Vol. \textbf{269} (2020), 6941--6987.

\bibitem{Stuart:1993} C. A. Stuart: {\em Guidance properties of nonlinear planar waveguides}, Arch. Rational Mech. Anal. {\bf 125} (1993), no. 1, 145--200.


\bibitem{Turitsyn} S. K. Turitsyn: {\em Spatial dispersion of nonlinearity and stability of multidimensional solitons}, Theor. Math. Phys. (United States) {\bf 64} (1986), no. 2.


\end{thebibliography}
\end{document}